\documentclass[12pt, a4paper, reqno]{amsart}
\usepackage{enumitem, amsmath, amsthm, amssymb, bbm, amsbsy, textgreek, mathrsfs, graphicx}
\usepackage[title]{appendix}
\usepackage[colorlinks=true]{hyperref}
\hypersetup{
  urlcolor     = cyan, %Colour for external hyperlinks
}

\newcommand{\abs}[1]{\left| #1\right|}

\newcommand{\spec}[1]{\spe(#1)}
\newcommand{\acspec}[1]{\spe_{ac}\left(#1\right)}
\newcommand{\essspec}[1]{\spe_{ess}\left(#1\right)}
\renewcommand{\epsilon}{\varepsilon}

\newcommand{\Rplus}{\mathbb{R}_{+}}
\newcommand{\Zplus}{\mathbb{Z_{+}}}
\newcommand{\1}{\mathbbm{1}}
\newcommand{\m}{\boldsymbol{m}}

\newcommand{\bbC}{\mathbb{C}}
\newcommand{\bbN}{\mathbb{N}}

\newcommand{\bbR}{\mathbb{R}}
\newcommand{\bbT}{\mathbb{T}}
\newcommand{\bbZ}{\mathbb{Z}}

\renewcommand{\kappa}{\varkappa}
\renewcommand{\phi}{\varphi}
\renewcommand{\theta}{\vartheta}

\newcommand{\calA}{\mathcal{A}}
\newcommand{\calB}{\mathcal{B}}

\newcommand{\calF}{\mathcal{F}}
\newcommand{\calH}{\mathcal{H}}
\newcommand{\calJ}{\mathcal{J}}
\newcommand{\calK}{\mathcal{K}}

\newcommand{\frakM}{\mathfrak{M}}

\newcommand{\frakS}{\mathfrak{S}}
\newcommand{\fraks}{\mathfrak{s}}

\newcommand{\sfc}{\mathsf{c}}
\newcommand{\sfn}{\mathsf{n}}

\newcommand{\ld}{\mathsf{LogDens}}
\newcommand{\uld}{\overline{\mathsf{LogDens}}}
\newcommand{\lld}{\underline{\mathsf{LogDens}}}
\newcommand{\Sp}{\mathfrak{S}_{p}}
\newcommand{\Sinfty}{\frakS_{\infty}}
\renewcommand{\to}{\rightarrow}
\newcommand{\wh}[1]{\widehat{#1}}
\newcommand{\ltwo}{\ell^{2}(\Zplus)}
\newcommand{\Ltwo}{L^{2}(\bbT)}
\newcommand{\LTwo}{L^{2}(0,1)}
\newcommand{\Gam}{\mathit{\Gamma}}

\DeclareMathOperator{\arcsech}{arcsech}

\DeclareMathOperator{\diag}{diag}
\DeclareMathOperator{\im}{Im}

\DeclareMathOperator{\M}{M}
\DeclareMathOperator{\singsupp}{sing\,supp}
\DeclareMathOperator{\supp}{supp}
\DeclareMathOperator{\Tr}{Tr}

\DeclareMathOperator{\arctanh}{arctanh}

\DeclareMathOperator{\sign}{sign}
\DeclareMathOperator{\spe}{spec}

\renewcommand{\Im}{\im}

\newcommand{\eps}{\epsilon}

\theoremstyle{plain}
\newtheorem{theorem}{Theorem}[section]
\newtheorem{lemma}[theorem]{Lemma}
\newtheorem{proposition}[theorem]{Proposition}
\newtheorem{corollary}[theorem]{Corollary}

\theoremstyle{definition}

\newtheorem{example}[theorem]{Example}
\newtheorem*{example*}{Example}

\theoremstyle{remark}
\newtheorem{remark}[theorem]{Remark}

\numberwithin{equation}{section}

\begin{document}

\title[Spectral density of Hankel operators]{The spectral density of Hankel operators with piecewise continuous symbols}
\author{Emilio Fedele}
\address{Department of Mathematics, King's College London, WC2R 2LS, London.}
\email{\href{mailto:emilio.fedele@kcl.ac.uk}{\nolinkurl{emilio.fedele@kcl.ac.uk}}}
\maketitle
\begin{abstract}
In 1966, H. Widom proved an asymptotic formula for the distribution of eigenvalues of the $N\times N$ truncated Hilbert matrix for large values of $N$. In this paper, we extend this formula to Hankel matrices with symbols in the class of piece-wise continuous functions on the unit circle. Furthermore, we show that the distribution of the eigenvalues is independent of the choice of truncation (e.g. square or triangular truncation).
\end{abstract}
%%%%%%%%%%%%%%%%%%%%%%%%%%%%%%%%%%%%%%%%%%%%%%%%%%%%%%%%%%%%%%%%%
\begin{section}{Introduction}
\begin{subsection}{General setting and first results.}
Given an (essentially) bounded function $\omega$, called a \textit{symbol}, on  the unit circle $ \bbT=\{v \in \bbC\ \vert\ \abs{v}=1\}$, the associated Hankel matrix, $\Gam(\wh{\omega})$, is the (bounded) operator on $\ltwo$, $\bbZ_{+}=\{0,1,2,\ldots\}$, whose ``matrix'' entries are $$\Gam(\wh{\omega})_{j,k}=\wh{\omega}(j+k),\quad j, k \geq 0,$$ where $\wh{\omega}$ denotes the sequence of Fourier coefficients $$\wh{\omega}(j)=\int_{0}^{2\pi}\omega(e^{ i \theta})e^{- i j\theta}\frac{d\theta}{2\pi},\ \ \  j \in \bbZ.$$ The matrix $\Gam(\wh{\omega})$ is always symmetric. In particular, it is self-adjoint if and only if $\wh{\omega}$ is real-valued. For instance, this is the case when $\omega$ satisfies the following symmetry condition
\begin{gather}\label{eqn:symmcond}
\omega(v)=\overline{\omega(\overline{v})},\ \ \ v \in \bbT. 
\end{gather}
In this paper, we consider symbols in the class of the piece-wise continuous functions on $\bbT$, denoted by $PC(\bbT)$, i.e. those symbols $\omega$ for which the limits 
\begin{equation}\label{eqn:limits}
    \omega(z+)=\lim_{\eps\to 0+}\omega(z e^{i \eps}),\ \ \ \omega(z-)=\lim_{\eps\to 0+}\omega(z e^{-i \eps}),
\end{equation}
 exist and are finite for all $z \in \bbT.$ The points $z \in \bbT$ for which the quantity $$\kappa_{z}(\omega)=\frac{\omega(z+)-\omega(z-)}{2}\neq 0$$ are called the \textit{jump discontinuities}  of $\omega$ and $\kappa_{z}(\omega)$ is the \textit{half-height of the jump} of the symbol at $z$. Due to the presence of jump discontinuities, Hankel matrices with these symbols are non-compact. The compactness of $\bbT$ and the existence of the limits in \eqref{eqn:limits} can be used to show that the sets $$\Omega_{s}=\{z \in \bbT \vert \abs{\kappa_{z}(\omega)}>s\}, \ \ \ s>0,$$ are finite and so the set of jump-discontinuities of $\omega$, denoted by $\Omega$, is at most countable. Furthermore, if the symbol satisfies \eqref{eqn:symmcond}, $\Omega$ is symmetric with respect to the real axis and for any $z \in \Omega$ $$\kappa_{\overline{z}}(\omega)=-\overline{\kappa_{z}(\omega)},$$ whereby we obtain that  $\abs{\kappa_{z}(\omega)}=\abs{\kappa_{\overline{z}}(\omega)}$, and at $z=\pm 1$, $\kappa_{z}(\omega)$ is purely imaginary.

Hankel matrices with piece-wise continuous symbols still attract attention in both the operator theory and spectral theory community, see for instance \cite{P-Y-Spectral,P-Y-Spec} and references therein. S. Power, \cite{Pow-Disc}, showed that the essential spectrum of such matrices consists of bands depending only on the heights of the jumps of the symbol and gave the following identity:
\begin{align}\notag
 \essspec{\Gam(\wh{\omega})}=\left[0, -i\kappa_{1}(\omega)\right]&\cup\left[0, -i\kappa_{-1}(\omega)\right]\cup\\ \label{eqn:essspec}
  &\bigcup_{z \in \Omega\backslash \{\pm 1\}}\left[-i(\kappa_{z}(\omega)\kappa_{\overline{z}}(\omega))^{1/2},i(\kappa_{z}(\omega)\kappa_{\overline{z}}(\omega))^{1/2}\right],
\end{align} 
where the notation $[a,b], a,b \in \bbC$ denotes the line segment joining $a$ and $b$. Assuming that the symbol has finitely many jumps and, say, it is Lipschitz continuous on the left and on the right of the jumps, in \cite{P-Y-Spectral}, a more detailed picture is obtained for the absolutely continuous (a.c.) spectrum of $\abs{\Gam(\wh{\omega})}=\sqrt{\Gam(\wh{\omega})^{\ast}\Gam(\wh{\omega})}$, where the following formula is obtained $$\acspec{\abs{\Gam(\wh{\omega})}}=\bigcup_{z \in \Omega}\left[0, \abs{\kappa_{z}(\omega)}\right].$$ Furthermore, it is shown that each band contributes 1 to the multiplicity of the a.c. spectrum.

\begin{example*} 
First examples of symbols fitting in this scheme are the following 
\begin{gather}\label{eqn:hilbsymb}
\gamma(e^{i\theta})=i\pi^{-1} e^{-i\theta}(\pi-\theta),\ \ \psi(e^{i\theta})=2 \1_{E}(e^{i\theta}),\quad \theta \in [0, 2\pi).
\end{gather}
where $\1_{E}$ is the characteristic function of the set $E=\{\theta \in\ [0,\, 2\pi):\,\cos\theta>0\}$. It is clear that both $ \gamma, \psi \in PC(\bbT)$, and their jumps occur at $z=1$ and $z=\pm i$ respectively and $\kappa_{1}(\gamma)=i$, $\kappa_{\pm i}(\psi)=\mp 1$. Simple integration by parts shows that $$\wh{\gamma}(j)=\frac{1}{\pi(j+1)},\ \ \wh{\psi}(j)=\frac{2\sin(\pi j/2)}{\pi j}, \quad j\geq 0,$$ with the understanding that $\wh{\psi}(0)=1$. Power's result in \eqref{eqn:essspec} in these cases gives \begin{gather}\label{eqn:hilbsequences}
    \essspec{\Gam(\wh{\gamma})}=\left[0,\, 1\right], \quad \essspec{\Gam(\wh{\psi}\, )}=\left[-1,\, 1\right].
\end{gather} The matrix $\Gam(\wh{\gamma})$, known as the Hilbert matrix, has simple a.c. spectrum coinciding with the interval $[0,1]$ and a full spectral decomposition was exhibited in \cite{Ros}.  In \cite{kostrykin2008krein}, the authors perform a more detailed spectral analysis of $\Gam(\wh{\psi})$ and show that its spectrum is purely a.c. of multiplicity one and coincides with the interval $[-1,\, 1]$.
\end{example*}

For $N\geq 1$, let $\Gam^{(N)}(\wh{\omega})$ be the $N\times N$ Hankel matrix $$\Gam^{(N)}(\wh{\omega})=\{\wh{\omega}(j+k)\}_{j,k= 0}^{N-1}.$$ We wish to give a description of the relationship between the spectrum of the infinite matrix $\Gam(\wh{\omega})$ and that of its truncation $\Gam^{(N)}(\wh{\omega})$. More specifically:
\begin{enumerate}[label=(\roman*)]
    \item for a non-self-adjoint Hankel matrix, we study the distribution of the singular values of $\Gam^{(N)}(\wh{\omega})$ inside the spectrum of $\abs{\Gam(\wh{\omega})}$;
    \item in the self-adjoint setting, we study the distribution of the eigenvalues of $\Gam^{(N)}(\wh{\omega})$ inside the spectrum of $\Gam(\wh{\omega})$.
\end{enumerate}
To do so, for a non-self-adjoint Hankel matrix $\Gam(\wh{\omega})$ we study the asymptotic behaviour of the \textit{singular-value counting function}
\begin{gather*}
\sfn(t;\Gam^{(N)}(\wh{\omega}))=\#\{n:\ s_{n}(\Gam^{(N)}(\wh{\omega}))>t\},\ \ \ t>0, 
\end{gather*}
as $N \to \infty$. Here $\{s_{n}(\Gam^{(N)}(\wh{\omega}))\}_{n\geq 1}$ is the sequence of singular values of $\Gam^{(N)}(\wh{\omega})$. In particular, we study the \textit{logarithmic spectral density} of $\abs{\Gam(\wh{\omega})}$, defined as
\begin{gather}\label{eqn:lsd}
     \ld_{\square}(t;\Gam(\wh{\omega})):=\lim_{N \to \infty}\frac{\sfn(t; \Gam^{(N)}(\wh{\omega}))}{\log(N)}.
\end{gather}
 
\noindent For a self-adjoint $\Gam(\wh{\omega})$, its spectrum, $\spec{\Gam(\wh{\omega})}$, is a subset of the real line and so we look at how the positive and negative eigenvalues of $\Gam^{(N)}(\wh{\omega})$ distribute inside $\spec{\Gam(\wh{\omega})}$. To this end, we analyze the behaviour of the \textit{eigenvalue counting functions} $$\sfn_{\pm}(t;\Gam^{(N)}(\wh{\omega}))=\#\{n:\ \lambda^{\pm}_{n}(\Gam^{(N)}(\wh{\omega}))>t\},\ \ \ t>0,$$ as $N \to \infty$. Here $\{\lambda^{\pm}_{n}(\Gam^{(N)}(\wh{\omega}))\}_{ n\geq 1}$ are the sequences of positive eigenvalues of $\pm \Gam^{(N)}(\wh{\omega})$ respectively. In this setting, we study the functions
\begin{gather}\label{eqn:lsdsa}
     \ld_{\square}^{\pm}(t;\Gam(\wh{\omega})):=\lim_{N \to \infty}\frac{\sfn_{\pm}(t; \Gam^{(N)}(\wh{\omega}))}{\log(N)}.
\end{gather}
Similarly to the non-self-adjoint setting, we call the function $\ld^{+}_{\square}$ (resp. $\ld^{-}_{\square}$) in \eqref{eqn:lsdsa} the \textit{positive} (resp. \textit{negative}) \textit{logarithmic spectral density} of $\Gam(\omega)$. 

The $\square$ appearing as an index in the definitions of the logarithimic spectral densities in \eqref{eqn:lsd} and \eqref{eqn:lsdsa} has been chosen to stress the fact that, a priori, these quantities depend on our choice to truncate the infinite matrix $\Gam(\wh{\omega})$ to its upper $N\times N$ square. Furthermore, the terminology we use for the functions $\ld_{\square}$ and $\ld^{\pm}_{\square}$ comes from the fact that we are only studying a logarithmically-small portion of the singular values (or eigenvalues) of the matrix $\Gam^{(N)}(\wh{\omega})$. Their definitions are motivated by the results obtained by Widom (see \cite[Theorem 4.3]{Widom}) for the Hilbert matrix $\Gam(\wh{\gamma})$, where he showed that
\begin{gather}\label{eqn:limit1}
\ld_{\square}(t;\Gam(\wh{\gamma}))=\mathsf{c}(t),\\\label{eqn:limit2}
\ld_{\square}^{-}(t;\Gam(\wh{\gamma}))=0,\quad \ld^{+}_{\square}(t;\Gam(\wh{\gamma}))=\mathsf{c}(t).
\end{gather}
Here $\mathsf{c}(t):=0$ whenever $t \notin (0,1)$ and 
\begin{gather}\label{eqn:hilbdensity}
\mathsf{c}(t):=\frac{1}{\pi^{2}}\arcsech(t)=\frac{1}{\pi^{2}}\log\left(\frac{1+\sqrt{1-t^{2}}}{t} \right), \quad t \in (0,1].
\end{gather}
We note that a factor of $2\pi$ is missing in the statement of \cite[Theorem 4.3]{Widom}. The aim of this paper is to extend \eqref{eqn:limit1} to a general symbol $\omega \in PC(\bbT)$. In particular, for a non-self-adjoint Hankel matrix, we aim to show that
\begin{gather}\label{eqn:poissonasym}
\ld_{\square}(t;\Gam(\wh{\omega}))=\sum_{z\in \Omega}\mathsf{c}\left( t\abs{\kappa_{z}(\omega)}^{-1}\right),
\end{gather}
where $\mathsf{c}$ is the function defined in \eqref{eqn:hilbdensity}. Recall that the symbol $\psi$ defined in \eqref{eqn:hilbsymb} has jumps at $\pm 1$ whose half-height is $\kappa_{\pm i}(\psi)=\mp 1$, so for the Hankel matrix $\Gam(\wh{\psi})$ the formula \eqref{eqn:poissonasym} yields
\begin{gather*}
\ld_{\square}(t;\Gam(\wh{\psi}))=2 \mathsf{c}\left(t\right).
\end{gather*}
 For self-adjoint Hankel matrices, we extend the result in \eqref{eqn:limit2} to symbols $\omega \in PC(\bbT)$ satisfying \eqref{eqn:symmcond} and obtain
\begin{align}\label{eqn:poissonasymsa}
\ld^{\pm}_{\square}(t;\Gam(\wh{\omega}))&=\mathsf{c}\left(t\abs{\kappa_{1}(\omega)}^{-1}\right)\1_{\pm}(-i\kappa_{1}(\omega))+\mathsf{c}\left(t\abs{\kappa_{-1}(\omega)}^{-1}\right)\1_{\pm}(-i\kappa_{-1}(\omega))\cr
&+\sum_{z\in \Omega^{+}}{\mathsf{c}\left(t\abs{\kappa_{z}(\omega)}^{-1}\right)},
\end{align}
where $\Omega^{+}=\{z\in \Omega \ \vert \Im z>0\}$, and $\1_{\pm}$ is the indicator function of  the half-line $(0, \pm \infty)$. Again, the function $\mathsf{c}$ has been defined in \eqref{eqn:hilbdensity}.  In particular, for the symbol $\psi$ in \eqref{eqn:hilbsymb}, we obtain that $$\ld^{\pm}_{\square}(t;\Gam(\wh{\psi}))=\mathsf{c}\left(t\right).$$

\noindent A natural question that we also address here is that of the universality of  the limits in \eqref{eqn:lsd} and \eqref{eqn:lsdsa}. In other words, we investigate whether they depend on the choice of ``regularisation'' of the matrix $\Gam(\wh{\omega})$. For instance, the main results of this paper, Theorems \ref{thm:sd} and \ref{thm:sdsa} below, tell us that the singular values of the matrix $\Gam^{(N)}(\wh{\omega})$ and of the regularised matrix
\begin{align}\label{eqn:poistrun}
\Gam_{N}(\wh{\omega})=\left\{e^{-\frac{j+k}{N}}\wh{\omega}(j+k)\right\}_{j,k\geq 0},\quad N\geq 1,
\end{align}
have the same distribution for large values of $N$. 
\end{subsection}
%%%%%%%%%%%%%%%%%%%%%%%%%%%%%%%%%%%%%%%%%%%%%%%%%%%%%%%%%%%%%%%%%
\begin{subsection}{Schur-Hadamard multipliers}
The truncation of a matrix to its finite $N\times N$ upper block and the ``matrix regularisation" in \eqref{eqn:poistrun} are examples of Schur-Hadamard multipliers, defined below.

\noindent For a bounded sequence $(\tau(j,k))_{j,k\geq 0}$, called a \textit{multiplier}, and a bounded operator $A$ on $\ltwo$, the Schur-Hadamard multiplication of $\tau$ and $A$ is the operator on $\ltwo$, $\tau\star A$, formally defined through the quadratic form
\begin{gather}\label{eqn:schurmultiplication}
((\tau\star A)e_{j}, e_{k})=\tau(j,k)(Ae_{j},e_{k}),\ \ \ j, k\geq 0,
\end{gather}
where $e_{j}$ is the $j$-th vector of the standard basis of $\ltwo$. Various authors in the literature, \cite{Bennett, birman1967stieltjes, Peller-Hankel}, have addressed the issue of establishing how properties of $\tau$ translate into the boundedness of this operation on the space of bounded operators and the Schatten classes $\Sp$ (for a definition see section 2 below). To do so, they have studied the operator norms
\begin{align}\label{eqn:norm}
\|\tau\|_{\frakM}&=\sup_{\|A\|=1}\|\tau\star A\|,\\ \label{eqn:pnorm}
\|\tau\|_{\frakM_{p}}&=\sup_{\|A\|_{\Sp}=1}\|\tau\star A\|_{\Sp},\quad 1\leq p\leq \infty.
\end{align}
Using the duality of $\Sp$-classes, it is possible to show that the following identities hold 
\begin{align}\label{eqn:duality1}
\|\tau\|_{\frakM_{1}}&=\|\tau\|_{\frakM_{\infty}}=\|\tau\|_{\frakM},\\\label{eqn:duality2}
\|\tau\|_{\frakM_{p}}&=\|\tau\|_{\frakM_{q}},\quad 1< p,\, q<\infty,\, p^{-1}+q^{-1}=1,
\end{align}
and so it is sufficient to study the boundedness of Schur-Hadamard multiplication on $\Sp$ for $p\geq 2$. The case of $p=2$ is somewhat trivial. In fact, the structure of $\frakS_{2}$ gives that any bounded sequence $\tau$ is a bounded Schur-Hadamard multiplier and furthermore $$\|\tau\|_{\frakM_{2}}=\sup_{j,k\geq 0}\abs{\tau(j,k)}.$$ For a general $1<p<\infty,\, p\neq 2$ not much is known with regards to the finiteness of $\|\tau\|_{\frakM_p}$. However, a necessary and sufficient condition for the boundedness of a Schur-Hadamard multiplier on the space of bounded operators (and, as a consequence of \eqref{eqn:duality1}, on $\frakS_{1}$ and $\Sinfty$) is known and can be found in \cite{birman2003double}. 

For the purposes of this paper, we will consider the Schur-Hadamard multiplier $\tau$ in \eqref{eqn:schurmultiplication} as the restriction to $\bbZ_{+}^{2}$ of a bounded function defined on $[0,\, \infty)^{2}$. For $N\geq 1$ set $\tau_{N}(j,k)=\tau(jN^{-1}, kN^{-1})$. If $\tau$ is such that the sequence of $\tau_{N}$ satisfies the following
\begin{equation}\label{eqn:uniform}
\sup_{N\geq 1}\|\tau_{N}\|_{\frakM}<\infty,
\end{equation}
we say that $\tau$ \textit{induces a uniformly bounded multiplier}. An easy example of such a multiplier is the $ N\times N$ truncation of an infinite matrix. To see this take the function 
\begin{align}\label{eqn:square}
\tau_{\square}(x,y)=\1_{\square}(x,y),
\end{align}
where $\1_{\square}$ is the characteristic function of the half-open unit square $[0,1)^{2}$. For any bounded operator $A$, $\tau_{N}\star A$ is the truncation to its upper $N\times N$ block and so we have that for any $N \geq 1$ $$\|\tau_{N}\|_{\frakM}= 1.$$ We discuss some more examples of Schur-Hadamard multipliers below. 
\end{subsection}
%%%%%%%%%%%%%%%%%%%%%%%%%%%%%%%%%%%%%%%%%%%%%%%%%%%%%%%%%%%%%%%%%
\begin{subsection}{Statement of the main results.} 
As we anticipated, our main results are not only concerned to the existence of the limits in \eqref{eqn:limit1} and \eqref{eqn:limit2}, but also with their universality. In other words, for a Hankel matrix $\Gam(\wh{\omega})$ and  a given multiplier $\tau$, we show that under some mild assumptions on $\tau$, see \ref{(A)}-\ref{(C)} below, the function
\begin{equation}
    \ld_{\tau}(t;\Gam(\wh{\omega})):=\lim_{N\to \infty}\frac{\sfn(t;\tau_{N}\star\Gam(\wh{\omega}))}{\log N},\quad t>0,
\end{equation}
is independent of the choice of $\tau$. Similarly, for a self-adjoint Hankel matrix and a multiplier $\tau$ such that $\tau(x,y)=\overline{\tau(y,x)}$, we show that the same is true for the functions
\begin{equation}
    \ld_{\tau}^{\pm}(t;\Gam(\wh{\omega})):=\lim_{N\to \infty}\frac{\sfn_{\pm}(t;\tau_{N}\star\Gam(\wh{\omega}))}{\log N},\quad t>0.
\end{equation}
Note that when $\tau=\tau_{\square}$ as in \eqref{eqn:square}, the functions $\ld_{\tau}(t;\Gam(\wh{\omega}))$ and $ \ld_{\tau}^{\pm}(t;\Gam(\wh{\omega}))$ are precisely those defined in \eqref{eqn:limit1} and \eqref{eqn:limit2}.

\noindent Let us state the following assumptions on $\tau$:
\begin{enumerate}[label=(\Alph*)]
\item\label{(A)} $\tau$ induces a uniformly bounded Schur-Hadamard multiplier, i.e. \eqref{eqn:uniform} holds;
\item\label{(B)} $\tau(0,0)=1$ and for some $\eps>0$ and some $\beta >1/2$, there exists $C_{\beta}>0$, so that  $$\abs{\tau(x, y)-1}\leq C_{\beta}\abs{\log(x+y)}^{-\beta}, \quad \forall\, 0\leq x,\,y\leq \eps;$$
\item\label{(C)} for some $\alpha>1/2$ one can find $C_{\alpha}$ so that $$\abs{\tau(x,y)}\leq C_{\alpha}\log(x+y+2)^{-\alpha},\ \ \ \forall\, x,\, y\geq 0.$$ 
\end{enumerate}
Then \eqref{eqn:poissonasym} is a particular case of the following:
\begin{theorem}\label{thm:sd}
Let $\tau$ be a multiplier satisfying \textnormal{\ref{(A)}-\ref{(C)}}. Let $\omega \in PC(\bbT)$ and $\Omega$ be the set of its discontinuities. Then
\begin{gather}\label{eqn:sdho}
\ld_{\tau}(t;\Gam(\wh{\omega}))=\sum_{z\in \Omega}\mathsf{c}\left( t\abs{\kappa_{z}(\omega)}^{-1}\right)
\end{gather}
where $\mathsf{c}(t)$ is the function defined in \eqref{eqn:hilbdensity}.
\end{theorem}
\noindent Analogously for the self-adjoint case, \eqref{eqn:poissonasymsa} is a particular case of the Theorem below:
\begin{theorem}\label{thm:sdsa}
Let $\tau$ satisfy conditions \textnormal{\ref{(A)}-\ref{(C)}} and such that $\tau(x, y)=\overline{\tau(y,x)}.$ Suppose $\omega \in PC(\bbT)$ satisfies \eqref{eqn:symmcond} and let $\Omega^{+}=\{z \in \Omega\, \vert\, \im z>0\}$. Then 
\begin{align}\label{eqn:sdsa}
\ld_{\tau}^{\pm}(t;\Gam(\wh{\omega}))&=\sum_{z\in \Omega^{+}}{\mathsf{c}\left(t\abs{\kappa_{z}(\omega)}^{-1}\right)}\cr
&+\mathsf{c}\left(t\abs{\kappa_{1}(\omega)}^{-1}\right)\1_{\pm}(-i\kappa_{1}(\omega))\cr
&+\mathsf{c}\left(t\abs{\kappa_{-1}(\omega)}^{-1}\right)\1_{\pm}(-i\kappa_{-1}(\omega)),
\end{align}
where $\mathsf{c}(t)$ is the function defined in \eqref{eqn:hilbdensity} and $\1_{\pm}$ is the characteristic function of the half-line $(0, \pm \infty)$.
\end{theorem}
\end{subsection}
%%%%%%%%%%%%%%%%%%%%%%%%%%%%%%%%%%%%%%%%%%%%%%%%%%%%%%%%%%%%%%%%%
\begin{subsection}{Remarks}
\begin{enumerate}[label=(\Alph*)]
    \item It is clear that Theorems \ref{thm:sd} and \ref{thm:sdsa} generalise the result of Widom in \cite{Widom} mentioned earlier in \eqref{eqn:hilbdensity} to any multiplier $\tau$ and, in both instances, we only describe the behaviour of a logarithmically small portion of the spectrum of $\tau_{N}\star\Gam(\wh{\omega})$ as most of the points lie in a vicinity of 0.
    \item Both Theorems \ref{thm:sd} and \ref{thm:sdsa} deal with a rather general class of symbols and for this reason we cannot say more about the error term in the asymptotic expansion of the functions $\sfn, \sfn_{\pm}$. In fact, we can only write $$\sfn(t; \tau_{N}\star\Gam(\wh{\omega}))=\log(N)\sum_{z \in \Omega}\mathsf{c}(t\abs{\kappa_{z}(\omega)}^{-1})+o(\log(N)),\quad N\to \infty.$$ If, however, we were to restrict our attention to those symbols with finitely many jumps and some degree of smoothness away from them (say Lipschitz continuity), we would obtain a more precise estimate,  see \cite{FedeleGebert}, however the trade-off would be that of making our results less general.
    \item Studying the spectral density of operators is common to many areas of spectral analysis. In particular, our results can be put in parallel to well-known results in the spectral theory of Schr\"odinger operators, where the existence and universality of the density of states is a well-studied problem for a wide class of potentials, see \cite{cycon2009schrodinger} and \cite[Section 5]{kirsch2007invitation} for an introduction and references therein for more on this subject. 
    \item Both Theorems \ref{thm:sd} and \ref{thm:sdsa} assume that the multiplier $\tau$ induces a uniformly bounded multiplier on the space of bounded operators. However, this condition can be substantially weakened in two different ways.
    
    \noindent Firstly, we can weaken assumption \ref{(A)} on the multiplier $\tau$ by assuming that  for some finite $p>1$, $\tau$ induces a uniformly bounded Schur-Hadamard multiplier on $\Sp$, or in other words that
    \begin{equation}\label{eqn:spuniform}
        \sup_{N\geq 1}\|\tau_{N}\|_{\frakM_{p}}=\sup_{N\geq 1}\left(\sup_{\|A\|_{\Sp}=1}\|\tau_{N}\star A\|_{\Sp}\right)<\infty.
    \end{equation}
    However, as a trade-off, we need to impose more stringent conditions on the symbol, as the following statement shows:
    \begin{proposition}\label{prop:schattenld}
    Suppose $\tau$ satisfies \eqref{eqn:spuniform} as well as Assumptions \ref{(B)} and \ref{(C)}. If the symbol $\omega$ can be written as 
    \begin{gather}\label{eqn:basicrep}
    \omega(v)=-i\sum_{z \in \Omega}\kappa_{z}(\omega)\gamma(\overline{z}v)+\eta(v),\quad v\in\bbT,
    \end{gather}
    where $\Omega$ is a finite subset of $\bbT$, $\gamma$ is the symbol in \eqref{eqn:hilbsymb} and $\eta$ is a symbol for which $\Gam(\wh{\eta}) \in \Sp$, then \eqref{eqn:sdho} holds. Furthermore, if $\tau(x,y)=\overline{\tau(y,x)}$ and $\omega$ also satisfies \eqref{eqn:symmcond}, then \eqref{eqn:sdsa} holds.
    \end{proposition}
    \noindent Secondly, we can assume that $\tau$ only induces a uniformly bounded Schur-Hadamard multiplier on the space bounded Hankel matrices, i.e. that
    \begin{equation}\label{eqn:hankeluniform}
        \sup_{N\geq 1}\left(\sup_{\|\Gam(\wh{\omega})\|=1}\|\tau_{N}\star \Gam(\wh{\omega})\|\right)<\infty.
    \end{equation}
     In this case, Theorems \ref{thm:sd} and \ref{thm:sdsa} still hold in their generality and we have the following
    \begin{proposition}\label{prop:hankelld}
     Let $\omega \in PC(\bbT)$ and let $\tau$ satisfy \eqref{eqn:hankeluniform} as well as Assumptions \ref{(B)} and \ref{(C)}. Then \eqref{eqn:sdho} holds. Furthermore, if $\omega$ satisfies the symmetry condition \eqref{eqn:symmcond} and $\tau(x,y)=\overline{\tau(y,x)}$, then \eqref{eqn:sdsa} holds.
    \end{proposition}
    \noindent We chose to make use of Assumption \ref{(A)} instead of \eqref{eqn:spuniform} and \eqref{eqn:hankeluniform}, because there are no known necessary and sufficient conditions for a multiplier to satisfy either of them. We give specific examples of multipliers that satisfy these conditions below. 
    
\end{enumerate}
\end{subsection}
%%%%%%%%%%%%%%%%%%%%%%%%%%%%%%%%%%%%%%%%%%%%%%%%%%%%%%%%%%%%%%%%%
\begin{subsection}{Some Examples of Schur-Hadamard multipliers}
\begin{example}[Factorisable multipliers]\label{example1}
If the function $\tau$ can be factorised as $$\tau(x,y)=f(x)g(y),\quad x,y\geq 0,$$ for some bounded function $f, g$, then it is easy to see that it induces a uniformly bounded Schur-Hadamard multiplier in the sense of \eqref{eqn:uniform}, and furthermore $$\sup_{N\geq 1}\|\tau_{N}\|_{\frakM}\leq\|f\|_{\infty}\|g\|_{\infty}.$$ As it was pointed out earlier in \eqref{eqn:square}, the truncation to the upper $N\times N$ square is an example of such a multiplier. Another example is given by choosing the function $\tau_{1}(x, y)=e^{-(x+y)}=e^{-x}e^{-y}$. This induces the regularisation in \eqref{eqn:poistrun} and it is immediate to see that $$\sup_{N\geq 1}\|(\tau_{1})_{N}\|_{\frakM}=1.$$ Furthermore, $\tau_{1}$ satisfies the assumptions \ref{(B)} and \ref{(C)} and so Theorems \ref{thm:sd} and \ref{thm:sdsa} hold.
\end{example}
\begin{example}[Non-examples]\label{example2} In stark contrast to the square truncation in \eqref{eqn:square}, the so-called ``main triangle projection" induced by the function 
\begin{equation}\label{eqn:Dirichlet}
\tau_{2}(x,y)=\1_{[0,1)}\left(x+y\right)
\end{equation}
is not uniformly bounded on the bounded operators, see \cite{A-C-N, Bonami-Bruna, K-P}, where it was shown that $$\sup_{\|A\|=1}\|(\tau_{2})_{N}\star A\|= \pi^{-1}\log(N)+o(\log(N)),\ \ \ N \to \infty.$$ However, $\tau_{2}$ is uniformly bounded on any Schatten class $\Sp,\, 1<p<\infty$, see \cite{birman2003double}, and so Proposition \ref{prop:schattenld} holds.

\noindent Proposition \ref{prop:hankelld} shows that Theorems \ref{thm:sd} and \ref{thm:sdsa} still hold in the case that the Schur-Hadamard multiplier is only uniformly bounded on the set of bounded Hankel matrices. An example of such a multiplier is given by the indicator function, $\tau_{\beta, \gamma}$, of the region $$\Xi_{\beta, \gamma}=\{(x, y)\in [0,1]^{2}\ \vert\ x\leq -\beta y+ \gamma\}, \quad \beta,\, \gamma \in \bbR.$$ Even though $\tau_{\beta, \gamma}$ does not induce, in general, a uniformly bounded Schur-Hadamard multiplier, it has been shown in \cite[Theorem 1(a)]{Bonami-Bruna} that this is the case on the set of bounded Hankel matrices for $\beta \neq 1, 0$ and any $\gamma$ (at $\beta=1$ and $\gamma=1$, $\tau_{1, 1}$ reduces to the multiplier $\tau_{2}$ considered above). With this at hand, an appropriate choice of the parameters $\beta$ and $\gamma$ gives \eqref{eqn:sdho} and \eqref{eqn:sdsa}.
\end{example}
\begin{example}[General Criterion]
For more complicated functions, the following criterion can be of help. Let $\Sigma \subset \bbR$ and $m$ be a measure on $\Sigma$. Suppose that for the function $\tau$ we can write $$\tau_{N}(j,k)=\tau\left(\frac{j}{N}, \frac{k}{N}\right)=\int_{\Sigma}e^{-it(j+k)}f_{N}(t)dm(t),\quad \forall j,\, k\geq 0$$ for some functions $f_{N} \in L^{1}(\Sigma, m)$ so that $\sup_{N\geq 1}\|f_{N}\|_{L^{1}(\Sigma, m)}<\infty$. It is not hard to check that $\tau$ induces a uniformly bounded Schur-Hadamard multiplier and that $$\sup_{N\geq 1}\|\tau_{N}\|_{\frakM}\leq \sup_{N\geq 1}\|f_{N}\|_{L^{1}(\Sigma, m)}.$$ Using this it is possible to show that the function
\begin{gather}\label{eqn:cesaro}
\tau_{3}(x,y)=\left(1-(x+y)\right)\1_{[0,1)}\left(x+y\right),
\end{gather}
induces a uniformly bounded Schur-Hadamard multiplier, since one has the following representation
\begin{gather}\label{eqn:fejer}
\tau_{3}(j N^{-1},k N^{-1})=\int_{0}^{2\pi}e^{-i t \left(j+k\right)}F_{N}(e^{it})\frac{dt}{2\pi},\quad j,\, k\geq 0,
\end{gather}
where $F_{N}$ in \eqref{eqn:fejer} denotes the $N$-th Fej\'er kernel.

\noindent The multipliers induced by the functions 
\begin{gather*}
\tau_{1}(x,y)=e^{-(x+y)},\\
\tau_{2}(x,y)=\1_{[0,1)}(x+y),\\
\tau_{3}(x,y)=\left(1-(x+y)\right)\1_{[0,1)}\left(x+y\right)
\end{gather*}
are related to the Abel-Poisson, Dirichlet and Cesaro summation methods respectively and share some of the properties of the operators of convolution with the respective kernels, see Section 3 for more on the Poisson kernel.
\end{example}
\end{subsection}
%%%%%%%%%%%%%%%%%%%%%%%%%%%%%%%%%%%%%%%%%%%%%%%%%%%%%%%%%%%%%%%%%
\begin{subsection}{Outline of the proofs}
To prove Theorems \ref{thm:sd} and \ref{thm:sdsa} we use a similar approach to the one in \cite{Pow-Disc} and \cite{P-Y-Localization} and combine abstract results concerned with the general properties of the functions $\sfn, \sfn_{\pm}$ (see Section 2) and more hands-on function theoretic ones that are specific to the theory of Hankel matrices, see Section 3. 

To prove Theorem \ref{thm:sd}, we firstly assume that the set of jump-discontinuities, $\Omega$, of the symbol $\omega$ is finite and we write
\begin{gather}\label{eqn:rep1}
    \omega(v)=-i\sum_{z \in \Omega}\kappa_{z}(\omega)\gamma(\overline{z}v)+\eta(v),\quad v\in\bbT,
\end{gather}
where $\gamma$ is the symbol in \eqref{eqn:hilbsymb} and $\eta$ is a continuous function on $\bbT$. The analysis of $\ld_{\tau}(t;\Gam(\wh{\omega}))$ then proceeds with the study of each summand appearing in \eqref{eqn:rep1} and the interactions this has with all the others. In particular, Assumption \ref{(A)} allows us to disregard the contribution coming from the matrix $\Gam(\wh{\eta})$, i.e. it gives that $$\ld_{\tau}(t;\Gam(\wh{\omega}))=\ld_{\tau}\left(t;\sum_{z\in \Omega}\kappa_{z}(\omega)\Gam(\wh{\gamma}_{z})\right),$$ where $\gamma_{z}(v)=-i\gamma(\overline{z}v),\, v \in \bbT$.
 The invariance of the functions $\ld_{\tau}$ with respect to the choice of multiplier, proved in Theorem \ref{thm:invprinc}, gives that $$\ld_{\tau}\left(t;\sum_{z\in \Omega}\kappa_{z}(\omega)\Gam(\wh{\gamma}_{z})\right)=\ld_{\tau_{1}}\left(t;\sum_{z\in \Omega}\kappa_{z}(\omega)\Gam(\wh{\gamma}_{z})\right)$$ where the multiplier $\sigma_{1}(x,y)=e^{-(x+y)}$ is given in the Example \ref{example1} above, and it is shown to induce the regularisation in \eqref{eqn:poistrun}, i.e. $(\tau_{1})_{N}\star\Gam(\wh{\omega})=\Gam_{N}(\wh{\omega})$. For the multiplier $\tau_{1}$, we explicitly show that the operators $\Gam(\wh{\gamma}_{z})$ are mutually ``almost orthogonal" in the sense that if $z\neq w \in \Omega$, then both  $$\Gam_{N}(\wh{\gamma}_{z})^{\ast}\Gam_{N}(\wh{\gamma}_{w}),\quad \Gam_{N}(\wh{\gamma}_{z})\Gam_{N}(\wh{\gamma}_{w})^{\ast}$$ are trace-class. From here, Theorem \ref{thm:thm2.2}, gives that each jump contributes independently, or in other words that we can write $$\ld_{\tau_{1}}\left(t;\sum_{z\in \Omega}\kappa_{z}(\omega)\Gam(\wh{\gamma}_{z})\right)=\sum_{z \in \Omega}\ld_{\tau_{1}}(t;\kappa_{z}(\omega)\Gam(\wh{\gamma}_{z})).$$ We note here that the above is another instance of the general fact that jumps occurring at different points of the unit circle contribute independently to the spectral properties of the operator $\Gam(\wh{\omega})$. For this reason, we follow the terminology used by the authors of \cite{P-Y-Localization} and we refer to this fact as the \textit{``Localisation Principle"}.
 
 \noindent Finally, using once again the Invariance Principle, Theorem \ref{thm:invprinc}, and the result of Widom in \eqref{eqn:hilbdensity}, we obtain the identity  \eqref{eqn:poissonasym} for a symbol $\omega$ with finitely many jumps.

The proof of Theorem \ref{thm:sdsa} roughly follows the same outline. However, instead of writing the symbol $\omega$ as in \eqref{eqn:rep1}, we make use of the symmetry of the set of jump-discontinuities, $\Omega$, to decompose it as follows
\begin{gather}
    \omega(v)=\kappa_{1}\gamma_{1}(v)+\kappa_{-1}\gamma_{-1}(v)-i\sum_{z \in \Omega^{+}}{\left(\kappa_{z}(\omega)\gamma_{z}(v)+\kappa_{\overline{z}}(\omega)\gamma_{\overline{z}}(v)\right)}+\eta(v),\quad v \in \bbT,
\end{gather}
where $\Omega^{+}=\{z \in \Omega\ \vert\ \im{z}>0\}$ and, as before, $\gamma_{z}(v)=-i\gamma(\overline{z}v)$ and $\eta$ is a continuous symbol on $\bbT$. The same strategy used in the proof of Theorem \ref{thm:sd} leads to the following identity 
\begin{align*}
        \ld_{\tau}^{\pm}(t;\Gam(\wh{\omega}))&=\ld_{\tau_{1}}^{\pm}(t;\kappa_{1}(\omega)\Gam(\wh{\gamma}_{1}))+\ld_{\tau_{1}}^{\pm}(t;\kappa_{-1}(\omega)\Gam(\wh{\gamma}_{-1}))\\
                                             &+\sum_{z \in \Omega^{+}}\ld_{\tau_{1}}^{\pm}(t;\kappa_{z}(\omega)\Gam(\wh{\gamma}_{z})+\kappa_{\overline{z}}(\omega)\Gam(\wh{\gamma}_{\overline{z}})).
\end{align*}
The fact that the jumps of $\omega$ are arranged symmetrically around $\bbT$ can be used to show that the positive and negative eigenvalues of the compact operator $$\kappa_{z}(\omega)\Gam_{N}(\wh{\gamma}_{z})+\kappa_{\overline{z}}(\omega)\Gam_{N}(\wh{\gamma}_{\overline{z}})$$ are arranged almost symmetrically around 0, in a sense that we will specify in Lemma \ref{lemma:lemma4.3}-\ref{item:item4.3.2}. Using Theorem \ref{thm:thm2.3}, we conclude that 
\begin{gather}\label{eqn:symm}
   \ld_{\tau_{1}}^{\pm}(t;\kappa_{z}(\omega)\Gam(\wh{\gamma}_{z})+\kappa_{\overline{z}}(\omega)\Gam(\wh{\gamma}_{\overline{z}}))=\ld_{\tau_{1}}(t;\kappa_{z}(\omega)\Gam(\wh{\gamma}_{z})). 
\end{gather}
 Using once again the result of Widom in \eqref{eqn:hilbdensity}, we arrive at \eqref{eqn:sdsa}. It is worth noting here that \eqref{eqn:symm} shows that if $\omega$ has jumps occurring at a pair of complex conjugate points, then the upper and lower logarithmic spectral densities, $\ld^{\pm}(t, \Gam(\wh{\omega}))$, contribute equally to the logarithmic spectral density of $\abs{\Gam(\wh{\omega})}$, we refer to this as the \textit{``Symmetry Principle"}, following the terminology used by the authors of \cite{P-Y-Asymptotics}.

Both Theorem \ref{thm:sd} and \ref{thm:sdsa} are then extended to the case of a symbol with infinitely-many jump-discontinuities using an approximation argument first presented by Power in \cite{Pow-Disc} and subsequently in \cite[Ch. 10, Thm. 1.10]{Peller}, see Section 4 below. 
\end{subsection}
%%%%%%%%%%%%%%%%%%%%%%%%%%%%%%%%%%%%%%%%%%%%%%%%%%%%%%%%%%%%%%%%%
\end{section}
%%%%%%%%%%%%%%%%%%%%%%%%%%%%%%%%%%%%%%%%%%%%%%%%%%%%%%%%%%%%%%%%%
\begin{section}{Abstract properties of the spectral density}
\begin{subsection}{First definitions and results}
Let $\Sinfty$ denote the ideal of compact operators. For any $p>0$, $\Sp$ denotes the ideal of compact operators whose singular values are $p$-summable and let $\frakS_{0}=\cap_{p>0}\Sp$.  For $p\geq 1$, the $\Sp$-norm is defined as 
 \begin{gather*}
 \|A\|_{\Sp}^{p}=\sum_{n=1}^{\infty}s_{n}(A)^{p},\quad 1\leq p<\infty.\\
 \|A\|_{\Sinfty}=\sup_{n}s_{n}(A),\quad p=\infty.
 \end{gather*}
 Here $\{s_{n}(A)\}_{n=1}^{\infty}$ is the sequence of singular values of $A$ ordered in a decreasing manner with multiplicities taken into account. All operators in this section are bounded operators acting on the space of square summable sequence $\ltwo$.

The functions $\sfn, \sfn_{\pm}$ were defined in the Introduction. It is clear that $\sfn(t; A)=\sfn(t; A^{\ast}),$ as the non-zero singular values of $A$ and $A^{\ast}$ coincide and, furthermore one has 
\begin{gather}\label{eqn:eqn2.3}
\sfn(t; A)=\sfn(t^{2}; A^{\ast}A), \quad  t>0.
\end{gather}
For any self-adjoint operator $A$, the functions $\sfn$ and $\sfn_{\pm}$ are linked via the following: $$\sfn(t;A)=\sfn_{+}(t;A)+\sfn_{-}(t;A), \quad  t>0.$$
The singular-value counting function of $K \in \Sp,$ satisfies the following simple estimate:
\begin{lemma}\label{lemma:lemma2.1}
Let $K \in \Sp$, $1\leq p< \infty$, then, for any $t>0$ one has $$\sfn(t; K)\leq \frac{\|K\|_{\frakS_p}^{p}}{t^{p}}.$$ If $K$ is self-adjoint, the same holds for the functions $\sfn_{\pm}(t;K)$.
\end{lemma}
We will also use the following inequalities, known as Weyl's inequalitites, see \cite[Thm. 9, Ch. 9]{BirSol1}:
\begin{lemma}[\textbf{Weyl Inequality}]\label{lemma:weyl} Let $A, B$ be compact operators and $0<s<t$, then
\begin{gather}\label{eqn:weyl}
\sfn(t;A+B)\leq \sfn(t-s;A)+\sfn(s;B),\\\label{eqn:weylsa}
\sfn_{\pm}(t;A+B)\leq \sfn_{\pm}(t-s;A)+\sfn_{\pm}(s;B),
\end{gather}
with the last inequality holding for self-adjoint operators.  
\end{lemma}

For a bounded  $\tau$ on $[0, \infty)^{2}$ we have already defined in the Introduction the meaning of $\tau_{N}\star A$. We have the following simple
\begin{lemma}\label{lemma:convergence}
Let $\tau$ be continuous at $(0,0)$ with $\tau(0,0)=1$ and suppose it satisfies Assumption \ref{(A)}. Then for any bounded operator $A$, $\tau_{N}\star A\to A$ as $N \to \infty$ in the strong operator topology. Furthermore, if $A$ is compact, the same is true in the operator norm.
\end{lemma}
\begin{proof}[Proof of Lemma]
Recall that Assumption \ref{(A)} implies that for any operator $A$ one has 
\begin{gather}\label{eqn:normestimate1}
\|\tau_{N}\star A\|\leq \sup_{N\geq 1}\|\tau_{N}\|_{\frakM}\|A\|.
\end{gather}
Let $e_{j},\, j\geq 0$ be the standard basis vectors of $\ltwo$. Using continuity of $\tau$ at $(0,0)$ and the fact that $\tau(0,0)=1$, a simple calculation shows $(\tau_{N}\star A)e_{j}\to Ae_{j}$ in $\ltwo$ and so we obtain that $(\tau_{N}\star A)x\to A x$ as $N \to \infty$ for any finite sequence $x \in \ltwo$. 

For any $x \in \ltwo$, the result follows from a standard $\eps/3$ argument. In particular, for $\eps>0$, we find a finite sequence $x_{\eps}$ so that $\|x-x_{\eps}\|_{2}<\eps/3$. Using the triangle inequality and \eqref{eqn:normestimate1}, together with the fact that $(\tau_{N}\star A)x_{\eps}\to A x_{\eps}$ we obtain the assertion.

If $A$ is compact, we have that for any given $\eps>0$ we can find a finite matrix $B$ so that $\|A-B\|<\eps$. For any finite matrix $B$, the convergence $\tau_{N}\star B\to B$ in the strong operator topology implies convergence in the operator norm, so for $N$ large we have $\|(\tau_{N}\star B)- B\|<\eps$. The triangle inequality now yields
\begin{align*}
\|(\tau_{N}\star A)-A\|&\leq \|\tau_{N}\star(A-B)\|+\|\tau_{N}\star B-B\|+\|B-A\|\\
&\leq(1+\sup_{N\geq 1}\|\tau_{N}\|_{\frakM})\|A-B\|+\eps\\
&\leq(2+\sup_{N\geq 1}\|\tau_{N}\|_{\frakM})\eps \qedhere
\end{align*}
\end{proof}
As a consequence, we have the following
\begin{lemma}\label{lemma:compactperturbations}
Let $K \in \Sinfty$ and $\tau$ be as in Lemma \ref{lemma:convergence}. Then for any $t>0$ one has $$\sfn(t;\tau_{N}\star K)=O_{t}(1),\quad N \to \infty.$$ If $\tau(x,y)=\overline{\tau(y,x)}$ and $K$ is self-adjoint, the same holds for the functions $\sfn_{\pm}$.
\end{lemma}
\begin{proof}[Proof of Lemma]
From Lemma \ref{lemma:convergence}, we have that $\tau_{N}\star K\to K$ in the operator norm and, in particular, for $\eps>0$ we can find $N$ suitably large so that $\|\tau_{N}\star K-K\|<\eps$, whereby it follows that $\sfn(\eps;\tau_{N}\star K-K)=0$. Using \eqref{eqn:weyl}, we obtain for $0<\eps<t$:
\begin{align*}
\sfn(t;\tau_{N}\star K)&\leq \sfn(t-\eps;K)= C_{t}.
\end{align*}
The proof in the self-adjoint case follows exactly the same reasoning.
\end{proof}
\noindent Define $\calB_{0}$ as the set of operators on $\ltwo$:
\begin{gather}\label{eqn:decaycond}
A \in \calB_{0} \quad \Longleftrightarrow \quad \ A_{j,k}= O\left(\frac{1}{j+k}\right),\ \ \ j,k \to \infty.
\end{gather}
Clearly $A \in \calB_{0}$ if and only if there exists a sequence $ a \in \ell^{\infty}(\bbN^{2})$ so that $$A_{j,k}=\frac{a_{j,k}}{\pi(j+k+1)},\qquad \forall j, k\geq 0.$$ From Hilbert inequality one obtains the estimate $\|A\|\leq \|a\|_{\ell^{\infty}},$ and so $A$ is also bounded. If the multiplier $\tau$ satisfies assumption \ref{(C)}, i.e. if for some $\alpha>1/2,$ one has $$\abs{\tau(x,y)}\leq \frac{C_{\alpha}}{\log(x+y+2)^{\alpha}},\ \ \ \forall x,\,y, $$ it is not difficult to see that when $A \in \calB_{0}$ one has that $\tau_{N}\star A \in\frakS_{2}$, since we have the following estimate
\begin{align*}
    \|\tau_{N}\star A\|_{\frakS_{2}}^{2}&=\sum_{j,k,\geq 0}\abs{\tau\left(\frac{j}{N},\,\frac{ k}{N}\right)A_{j,k}}^{2}\\
    &\leq C_{\alpha} \sum_{j,k\geq 0}\frac{1}{\log\left(\frac{j+k}{N}+2\right)^{2\alpha}(j+k+1)^{2}}<\infty.
\end{align*}
In particular, $\tau_{N}\star A$ is a compact operator for any given $N$ and so it makes sense to study how the functions $\sfn(t;\tau_{N}\star A)$ and $\sfn_{\pm}(t;\tau_{N} \star A)$ (whenever $A$ is self-adjoint and $\tau(x,y)=\overline{\tau(y,x)}$) behave for large $N$. To this end, it is useful to define the following two functionals
\begin{gather}\label{eqn:upp}
\uld_{\tau}(t; A):=\limsup_{N\to \infty}\frac{\sfn(t; \tau_{N}\star A)}{\log(N)},\ \ \ t>0,\\\label{eqn:low}
\lld_{\tau}(t; A):=\liminf_{N\to \infty}\frac{\sfn(t; \tau_{N}\star A)}{\log(N)},\ \ \ t>0.
\end{gather}
If $\uld_{\tau}(t;A)=\lld_{\tau}(t;A)$, we denote by $\ld_{\tau}(t;A)$ their common value. For a self-adjoint operator $A \in \calB_{0}$, we define the functionals $\uld_{\tau}^{\pm}(t;A),\, \lld_{\tau}^{\pm}(t;A)$ with the functions $\sfn_{\pm}$ replacing $\sfn$ in \eqref{eqn:upp} and \eqref{eqn:low} respectively and denote by $\ld_{\tau}^{\pm}(t; A)$ their common value, if it exists. 
\end{subsection}
%%%%%%%%%%%%%%%%%%%%%%%%%%%%%%%%%%%%%%%%%%%%%%%%%%%%%%%%%%%%%%%%%
\begin{subsection}{Invariance of spectral densities}
For a fixed operator $A \in \calB_{0}$, we wish to study the relation between the asymptotic behaviour of $\mathsf{n}(t;\tau_{N}\star A)$ for large $N$ and the Schur-Hadamard multiplier $\tau$. In particular, the result below tells us that the function $\mathsf{n}(t;\tau_{N}\star A)$ (as well as $\mathsf{n}_{\pm}(t;\tau_{N}\star A)$) asymptotically behaves independently of the multiplier $\tau$. We refer to this phenomenon as the \textit{Invariance Principle} and we state it as follows

\begin{theorem}[\textbf{Invariance Principle}]\label{thm:invprinc}
Suppose $\tau_{1}, \tau_{2}$ are multipliers satisfying assumptions \textnormal{\ref{(B)}} and \textnormal{\ref{(C)}}. Then for $A \in \calB_{0}$ and for $t>0$ one has that 
\begin{gather*}
\uld_{\tau_{1}}(t+0;A)\leq \uld_{\tau_{2}}(t;A)\leq\uld_{\tau_{1}}(t-0;A),\cr 
\lld_{\tau_{1}}(t+0;A)\leq \lld_{\tau_{2}}(t;A)\leq\lld_{\tau_{1}}(t-0;A).
\end{gather*}
Similarly, for a self-adjoint $A\in \calB_{0}$ and $\tau_{i}(x,y)=\overline{\tau_{i}(y,x)}$, then one has that 
\begin{gather*}
\uld_{\tau_{1}}^{\pm}(t+0;A)\leq \uld_{\tau_{2}}^{\pm}(t;A)\leq\uld_{\tau_{1}}^{\pm}(t-0;A),\cr 
\lld_{\tau_{1}}^{\pm}(t+0;A)\leq \lld_{\tau_{2}}^{\pm}(t;A)\leq\lld_{\tau_{1}}^{\pm}(t-0;A).
\end{gather*}
\end{theorem}
Before proving the result, let us prove the following auxiliary lemma.
\begin{lemma}\label{lemma:unifbound}
Let $\sigma$  satisfy Assumption \ref{(C)} and be such that $\sigma(0,0)=0$ and such that for some $\eps>0$ and some $\beta >1/2$, there exists $C_{\beta}>0$, so that 
\begin{gather}\label{eqn:decaycondition}
\abs{\sigma(x, y)}\leq C_{\beta}\abs{\log(x+y)}^{-\beta}, \quad \forall\, 0\leq x,\,y\leq \eps.
\end{gather}
For any $A \in \calB_{0}$, one has $\sigma_{N}\star A\in \frakS_{2}$ and furthermore there exists $C>0$, independent of $N$, such that $$\|\sigma_{N}\star A\|_{\frakS_{2}}\leq C.$$
\end{lemma}
\begin{proof}[Proof of Lemma]
We need to estimate the following quantity $$\|\sigma_{N}\star A\|_{\frakS_{2}}^{2}=\sum_{j,k\geq 0}\abs{\sigma\left(\frac{j}{N},\, \frac{k}{N}\right)}^{2}\abs{A_{j,k}}^{2}.$$
A modification of the integral test and the assumption that $A \in \calB_{0}$, shows that one can find $C>0$ so that
\begin{align*}
\|\sigma_{N}\star A\|_{\frakS_{2}}^{2}&\leq C \iint_{\Rplus^{2}}\frac{\abs{\sigma\left(\frac{x}{N},\, \frac{y}{N}\right)}^{2}}{(x+y+1)^{2}}dxdy\\ 
&=C \iint_{\Rplus^{2}}\frac{\abs{\sigma\left(s, t\right)}^{2}}{(s+t+1/N)^{2}}dsdt\quad (:=I_{N}),
\end{align*}
the last inequality follows from the change of variables $x=Ns,\, y=Nt$. Let $\Omega_{\eps}=\{(s, t) \in \Rplus^{2}\, \vert\, s^{2}+t^{2}<\eps\}$ and $\Omega_{\eps}^{c}=\Rplus^{2}\setminus \Omega_{\eps}$, then:
\begin{align*}
I_{N}&= \iint_{\Omega_{\eps}}\frac{\abs{\sigma\left(s, t\right)}^{2}}{(s+t+1/N)^{2}}dsdt \qquad (:=J_{1})\\
&+\iint_{\Omega_{\eps}^{c}}\frac{\abs{\sigma\left(s, t\right)}^{2}}{(s+t+1/N)^{2}}dsdt\qquad (:=J_{2})
\end{align*}
We will show that each summand is uniformly bounded. Since $\sigma$ satisfies \eqref{eqn:decaycondition}, it follows
\begin{align*}
J_{1}&\leq\frac{C_{\beta}}{\log(2)^{2}}\iint_{\Omega_{\eps}}\frac{1}{\log\left(s^{2}+t^{2}\right)^{2\beta}(s^{2}+t^{2})}dsdt\\
&\leq C \int_{0}^{\eps}\frac{1}{r\log(r)^{2\beta}}dr<\infty.
\end{align*}
The second inequality is a consequence of writing the integral in polar coordinates and, since $\beta>1/2$, the last integral is finite. Using \ref{(C)}, it follows that
\begin{align*}
J_{2}&\leq C \iint_{\Omega_{\eps}^{c}}\frac{dsdt}{(s+t)^{2}\log(s+t+2)^{2\alpha}}\\
&\leq C \int_{\eps}^{\infty}\frac{dx}{x\log(x+2)^{2\alpha}}<\infty.
\end{align*}
We have thus obtained that $I_{N}$ is uniformly bounded in $N$, whereby the assertion follows. 
\end{proof}
\begin{proof}[Proof of Theorem \ref{thm:invprinc}]
Write $A_{i}^{(N)}= (\tau_{i})_{N}\star A$, Weyl's inequality \eqref{eqn:weyl} states that 
\begin{align*}
\mathsf{n}(t;A_{1}^{(N)})&=\mathsf{n}(t+s-s;A_{1}^{(N)}+A_{2}^{(N)}-A_{2}^{(N)})\\
&\leq \mathsf{n}(t-s;A_{2}^{(N)})+\mathsf{n}(s;A_{1}^{(N)}-A_{2}^{(N)}),
\end{align*}
for any $0<s<t$. Swapping the roles of $A_{1}^{(N)}$ and $A_{2}^{(N)}$ in the above, we obtain 
\begin{gather*}
    \mathsf{n}(t+s;A_{2}^{(N)})-\mathsf{n}(s;A_{1}^{(N)}-A_{2}^{(N)})\leq\mathsf{n}(t;A_{1}^{(N)})\leq\mathsf{n}(t-s;A_{2}^{(N)})+\mathsf{n}(s;A_{1}^{(N)}-A_{2}^{(N)}),
\end{gather*}
 Lemmas \ref{lemma:lemma2.1} and \ref{lemma:unifbound} together imply that $$\mathsf{n}(s;A_{1}^{(N)}-A_{2}^{(N)})=O_{s}(1)$$ as $N\to \infty$, and so we obtain that 
\begin{gather*}
    \uld_{\tau_{1}}(t+s;A)\leq \uld_{\tau_{2}}(t;A)\leq \uld_{\tau_{1}}(t-s;A),\\
    \lld_{\tau_{1}}(t+s;A)\leq \lld_{\tau_{2}}(t;A)\leq \lld_{\tau_{1}}(t-s;A).
\end{gather*}
Sending $s\to 0$ gives the desired inequalities. In the self-adjoint setting, the same reasoning carries through once we replace the function $\sfn$ with the functions $\sfn_{\pm}$.
\end{proof}
\end{subsection}
%%%%%%%%%%%%%%%%%%%%%%%%%%%%%%%%%%%%%%%%%%%%%%%%%%%%%%%%%%%%%%%%%
\begin{subsection}{Almost symmetric and almost orthogonal operators} As mentioned in the Introduction, we will use the following two results which are similar, at least in spirit, to Theorems 2.2 in \cite{P-Y-Localization} and Theorem 2.7 in \cite{P-Y-Spec} and their proofs follow the same scheme. 

\noindent From now on, we make no assumptions on the uniform boundedness and smoothness of our multiplier $\tau$ and write $A^{(N)}=\tau_{N}\star A$. The first of the two results discussed below is about the interactions at the level of their spectral densities between two operators. Namely, if $A, B$ are bounded operators whose truncations $A^{(N)}, B^{(N)}$ are almost orthogonal in the sense that $$A^{(N)\ast}B^{(N)}\in \Sp,\qquad A^{(N)}B^{(N)\ast}\in \Sp,$$ for some $p\geq 1$ uniformly in $N$, then each of the logarithmic spectral densities of $\abs{A}$ and $\abs{B}$ contributes independently to the logarithmic spectral density of $\abs{A+B}$. Let us state the result as follows
\begin{theorem}\label{thm:thm2.2}
Let $A_{i}$, with $1\leq i\leq L$, be a family of operators such that for some $p\in [1, \infty)$ one has $$\sup_{N\geq 1}\| A^{(N)\ast}_{j}A_{k}^{(N)}\|_{\Sp}<\infty,\ \ \ \sup_{N\geq 1}\| A_{j}^{(N)}A^{(N)\ast}_{k}\|_{\Sp}^{p}<\infty,\ \ \ \forall\,  j\neq k.$$ Then, for $A=\sum_{j=1}^{L}A_{j}$ and for any $t>0$:
\begin{align}\label{eqn:uldinv}
\sum_{j=1}^{L}\uld_{\tau}(t+0; A_{j})\leq \uld_{\tau}\left(t; A\right)&\leq\sum_{j=1}^{L}\uld_{\tau}(t-0; A_{j}),\\\label{eqn:lldinv}
\sum_{j=1}^{L}\lld_{\tau}(t+0; A_{j})\leq \lld_{\tau}\left(t; A\right)&\leq\sum_{j=1}^{L}\lld_{\tau}(t-0; A_{j}).
\end{align}
If all $A_{j}$ are self-adjoint and $\tau(x,y)=\overline{\tau(y,x)}$, then we have 
\begin{align}\label{eqn:uldinvpm}
\sum_{j=1}^{L}\uld_{\tau}^{\pm}(t+0; A_{j})\leq \uld_{\tau}^{\pm}\left(t; A\right)&\leq\sum_{j=1}^{L}\uld_{\tau}^{\pm}(t-0; A_{j}),\\\label{eqn:lldinvpm}
\sum_{j=1}^{L}\lld_{\tau}^{\pm}(t+0; A_{j})\leq \lld_{\tau}^{\pm}\left(t; A\right)&\leq\sum_{j=1}^{L}\lld_{\tau}^{\pm}(t-0; A_{j}).
\end{align}
\end{theorem}

\begin{proof}[Proof of Theorem \ref{thm:thm2.2}]
We will prove only \eqref{eqn:uldinv}, since \eqref{eqn:lldinv},\eqref{eqn:uldinvpm} and \eqref{eqn:lldinvpm} follow the same line of reasoning. Put $\calH=\oplus_{i=1}^{L}\ltwo$ and define the block diagonal operator $\calA_{N}=\diag\{A_{1}^{(N)},\ldots,A_{L}^{(N)}\}$ such that $$\calA_{N}(f_{1},\ldots, f_{L})=(A_{1}^{(N)}f_{1},\ldots, A_{L}^{(N)}f_{L}).$$ Similarly, let $\calA=\diag\{A_{1}, \ldots, A_{L}\}$. Define the operator $\calJ:\calH\to\ltwo$ as $$\calJ(f_{1},\ldots,f_{L})=\sum_{j=1}^{L}f_{j}.$$
The operator $(\calJ\calA_{N})^{\ast}(\calJ\calA_{N})$ can be written as an $L\times L$ block-matrix of the form:
$$\begin{pmatrix}
A_{1}^{(N)\ast}A_{1}^{(N)}&A_{1}^{(N)\ast}A_{2}^{(N)}&\ldots &A_{1}^{(N)\ast}A_{L}^{(N)}\\
A_{2}^{(N)\ast}A_{1}^{(N)}&A_{2}^{(N)\ast}A_{2}^{(N)}&\ldots &A_{2}^{(N)\ast}A_{L}^{(N)}\\
\vdots &\vdots &\ddots &\vdots\\
A_{L}^{(N)\ast}A_{1}^{(N)}&A_{L}^{(N)\ast}A_{2}^{(N)}&\ldots &A_{L}^{(N)\ast}A_{L}^{(N)}
\end{pmatrix}.$$
Since the operator $\calA_{N}^{\ast}\calA_{N}$ is the block diagonal $L\times L$ matrix $$\begin{pmatrix}
A_{1}^{(N)\ast}A_{1}^{(N)}&0&\ldots &0\\
0&A_{2}^{(N)\ast}A_{2}^{(N)}&\ldots &0\\
\vdots &\vdots &\ddots &\vdots\\
0&0&\ldots &A_{L}^{(N)\ast}A_{L}^{(N)}
\end{pmatrix},$$ 
it is easy to see that the difference $(\calJ\calA_{N})^{\ast}(\calJ\calA_{N})-\calA_{N}^{\ast}\calA_{N}$ is the $L\times L$ matrix
$$\calK_{N}=\begin{pmatrix}
0&A_{1}^{(N)\ast}A_{2}^{(N)}&\ldots &A_{1}^{(N)\ast}A_{L}^{(N)}\\
A_{2}^{(N)\ast}A_{1}^{(N)}&0&\ldots &A_{2}^{(N)\ast}A_{L}^{(N)}\\
\vdots &\vdots &\ddots &\vdots\\
A_{L}^{(N)\ast}A_{1}^{(N)}&A_{L}^{(N)\ast}A_{2}^{(N)}&\ldots &0
\end{pmatrix}.$$
Furthermore, since the operators $A_{j}$ are so that $\sup_{N}\|A_{j}^{(N)\ast}A_{k}^{(N)}\|_{\Sp}$ is finite for all $j\neq k$, then $$\sup_{N\geq 1}\|\calK_{N}\|_{\Sp}<\infty.$$ 
Thus, Weyl inequality \eqref{eqn:weyl} gives
\begin{align*}
\sfn(t; \calJ\calA_{N})=\sfn(t^{2}; (\calJ\calA_{N})^{\ast}(\calJ\calA_{N}))&\leq \sfn(t^{2}-s; \calA_{N}^{\ast } \calA_{N})+\sfn(s; \calK_{N})\\
																					 &=\sum_{j=1}^{L}\sfn(t^{2}-s; A_{j}^{(N)\ast} A_{j}^{(N)})+\sfn(s; \calK_{N}),
\end{align*}
where in the second line we used the fact that $\calA_{N}^{\ast } \calA_{N}$ is diagonal and so $$\sfn(t^{2}-s; \calA_{N}^{\ast } \calA_{N})=\sum_{j=1}^{L}\sfn(t^{2}-s; A_{j}^{(N)\ast} A_{j}^{(N)}).$$ Just as in the proof of Theorem \ref{thm:invprinc}, we swap the roles of $\calJ\calA_{N}$ and $\calA_{N}$ and, using \eqref{eqn:eqn2.3}, we obtain $$\sum_{j=1}^{L}\sfn(\sqrt{t^{2}+s}; A_{j}^{(N)})-\sfn(s; \calK_{N})\leq\sfn(t; \calJ\calA_{N})\leq \sum_{j=1}^{L}\sfn(\sqrt{t^{2}-s}; A_{j}^{(N)})+\sfn(s; \calK_{N}).$$
Dividing by $\log(N)$, exploiting the sub-additivity of the $\limsup$ in conjunction with Lemma \ref{lemma:lemma2.1} and sending $s \to 0$, one gets  
\begin{align}\label{eqn:inequalities}
\sum_{j=1}^{L}\uld_{\tau}(t+0; A_{j})\leq \uld_{\tau}(t; \calJ\calA)\leq \sum_{j=1}^{L}\uld_{\tau}(t-0; A_{j}).
\end{align}
Recall now that we set $A=\sum_{j=1}^{L}{A_{j}}$. We have
\begin{gather*}
A^{(N)}A^{(N)\ast}=\sum_{j, k=1}^{L}A_{j}^{(N)}A_{k}^{(N)\ast},\cr
(\calJ\calA_{N})(\calJ\calA_{N})^{\ast}=\sum_{j=1}^{L}A_{j}^{(N)}A_{j}^{(N)\ast}.
\end{gather*}
Write $D_{N}=A^{(N)}A^{(N)\ast}-(\calJ\calA_{N})(\calJ\calA_{N})^{\ast}$, then from our assumptions it follows that $\sup_{N\geq 1}\|D_{N}\|_{\Sp}<\infty$ and using \eqref{eqn:eqn2.3} in conjunction with the Weyl inequality \eqref{eqn:weyl}, we obtain
\begin{align*}
\sfn(t; A^{(N)\ast})=\sfn(t^{2}; A^{(N)}A^{(N)\ast})&\leq \sfn(t^{2}-s; (\calJ\calA_{N})(\calJ\calA_{N})^{\ast})+\sfn(s;D_{N})\\
&=\sfn(\sqrt{t^{2}-s};(\calJ\calA_{N})^{\ast})+\sfn(s;D_{N}).\\
\end{align*}
Whereby we obtain that
\begin{gather*}
\uld_{\tau}(t; A)=\uld_{\tau}(t; A^{\ast})\geq \uld_{\tau}(t+0; (\calJ\calA)^{\ast})=\uld_{\tau}(t+0; (\calJ\calA)),\\
\uld_{\tau}(t; A)=\uld_{\tau}(t; A^{\ast})\leq \uld_{\tau}(t-0; (\calJ\calA)^{\ast})=\uld_{\tau}(t-0; (\calJ\calA)). 
\end{gather*}
The above, in conjunction with \eqref{eqn:inequalities}, gives the result.
\end{proof}

\noindent The second result applies to a self-adjoint operator $A$ and establishes a relation between $\uld_{\tau}^{+}(t;A)$ (resp. $\lld^{+}(t;A)$) and $\uld_{\tau}^{-}(t;A)$ (resp. $\lld^{-}(t;A)$).  More precisely, if a self-adjoint operator $A$ is so that its truncation $A^{(N)}$ is almost symmetric under reflection around 0, in the sense that for some unitary operator $U$ one has $$UA^{(N)}+A^{(N)}U\ \in\ \Sp$$ for some $p\geq 1$ uniformly in $N$, then its upper and lower logarithmic spectral densities contribute equally to the logarithmic spectral density of $\abs{A}$. In other words, the positive and negative eigenvalues of $\tau_{N}\star A$ accumulate to the spectrum of $A$ in the same way. We can formulate this as follows
\begin{theorem}\label{thm:thm2.3}
Let $A$ be a self-adjoint operator and let $\tau$ be such that $\tau(x,y)=\overline{\tau(y,x)}$. Suppose there exists a unitary operator $U$ for which $$\sup_{N\geq 1}\|UA^{(N)}+A^{(N)}U\|_{\Sp}<\infty$$ for some $p\geq 1$. Then for $t>0$
\begin{gather*}
\uld_{\tau}^{-}(t+0;A)\leq\uld_{\tau}^{+}(t;A)\leq\uld_{\tau}^{-}(t-0;A),\\
\lld_{\tau}^{-}(t+0;A)\leq\lld_{\tau}^{+}(t;A)\leq\lld_{\tau}^{-}(t-0;A)
\end{gather*}
In particular, we get that
 \begin{gather*}
\uld_{\tau}(t+0;A)\leq 2\uld_{\tau}^{\pm}(t;A)\leq\uld_{\tau}(t-0;A),\\
\lld_{\tau}(t+0;A)\leq 2\lld_{\tau}^{\pm}(t;A)\leq\lld_{\tau}(t-0;A).
\end{gather*}
\end{theorem}

\begin{proof}[Proof of Theorem \ref{thm:thm2.3}]
Write $$K_{N}=A^{(N)}+U^{\ast}A^{(N)}U.$$ By assumption $\sup_{N\geq 1}\|K_{N}\|_{\Sp}$ is finite. Furthermore, by Weyl inequality \eqref{eqn:weylsa}
\begin{align*}
\sfn_{\pm}(t; A^{(N)})&=\sfn_{\pm}(t; -U^{\ast}A^{(N)}U+K_{N})\\
									  &\leq \sfn_{\pm}(t-s; -U^{\ast}A^{(N)}U)+\sfn_{\pm}(s,K_{N})\\
									  &=\sfn_{\mp}(t-s; A^{(N)})+\sfn_{\pm}(s,K_{N}),
\end{align*}
where $0<s<t.$ In particular, this gives that
\begin{gather*}
    \sfn_{-}(t+s; A^{(N)})-\sfn_{-}(s,K_{N})\leq \sfn_{+}(t; A^{(N)})\leq \sfn_{-}(t-s; A^{(N)})+\sfn_{+}(s,K_{N}).
\end{gather*}
The result follows once we divide through by $\log(N)$, send $N \to \infty$ and use Lemma \ref{lemma:lemma2.1}.
\end{proof}
\end{subsection}
\end{section}
%%%%%%%%%%%%%%%%%%%%%%%%%%%%%%%%%%%%%%%%%%%%%%%%%%%%%%%%%%%%%%%%%%%%%%%%%%%%%%%%%%%%%%%%%%%%%%%%%%%%%%%%%%%%%%%%%%%%%%%%%%%%%%%%%%
\begin{section}{Hankel operators and the Abel summation method}
\begin{subsection}{Hankel operators}
In the Introduction, we defined Hankel matrices acting on $\ltwo$, equivalently they can also be defined as integral operators acting on $\Ltwo$.

\noindent Let $\bbT$ be the unit circle in the complex plane, and $\m$ the Lebesgue measure normalised to 1, i.e $d\m(z)=(2\pi iz)^{-1}dz$. Define the Riesz projection as 
\begin{gather}\label{eqn:eqn1.1}
P_{+}:\Ltwo\longrightarrow \Ltwo\cr
(P_{+}f)(v)=\lim_{\eps\to 0}\int_{\bbT}{\frac{f(z)z}{z-(1-\eps)v}d\m(z)},\quad v \in \bbT.
\end{gather}
For a symbol $\omega$, the \textit{Hankel operator} $H(\omega)$ is:
\begin{gather}
H(\omega ):\Ltwo \to \Ltwo\cr\label{defn:defn1.1}
H(\omega)f= P_{+}\omega JP_{+}f
\end{gather} 
where $J$ is the involution $Jf(v)=f(\overline{v})$ and, by a slight abuse of notation, $\omega$ denotes both the symbol and the induced  operator of multiplication on $\Ltwo$. We can immediately see that if $\omega$ satisfies \eqref{eqn:symmcond}, $H(\omega)$ is self-adjoint. Furthermore, it is easy to see that
\begin{equation}\label{eqn:normestimate}
\|H(\omega)\|\leq \|\omega\|_{\infty}.
\end{equation}

\noindent For any non-negative integers $j,k$, one has 
\begin{align*}
\left(H(\omega)z^{j}, z^{k}\right)_{\Ltwo}&=\left(P_{+}\omega JP_{+}z^{j}, z^{k}\right)_{\Ltwo}\\
&=\left(\omega\cdot z^{-j}, z^{k}\right)_{\Ltwo}=\widehat{\omega}(j+k),
\end{align*}
and so the matrix representation of $H(\omega)$ in the basis $\{z^{n}\}_{n\in \bbZ}$ is the block-matrix
$$\begin{pmatrix}
0&0\\
0&\Gam(\widehat{\omega})
\end{pmatrix},$$  with respect to the orthogonal decomposition $\Ltwo=H^{2}\oplus(H^{2})^{\perp}$, where $H^{2}$ is the closed linear span in $\Ltwo$ of the monomials $\{z^{n}\}_{n\geq 0}$. In other words, $H(\omega)$ and $\Gam(\widehat{\omega})$ are unitarily equivalent (modulo kernels) under the Fourier transform.

\noindent For $0<r<1$, let $P_{r}$ be the Poisson kernel, defined as
\begin{gather*}
P_{r}(v)=\sum_{j=-\infty}^{\infty}{r^{\abs{n}}v^{n}}=\frac{1-r^{2}}{\abs{1-rv}^{2}},\ \ \ v\in \bbT.
\end{gather*}
For $\omega_{r}=P_{r}\ast \omega$, we have the identity 
\begin{equation}\label{eqn:split}
    H(\omega_{r})=C_{r}H(\omega)C_{r},
\end{equation}
 where $C_{r}$ is the operator of convolution by $P_{r}$ on $\Ltwo$. Furthermore, it is unitarily equivalent (modulo kernels) to the Hankel matrix 
\begin{gather}\label{eqn:poissontruncation}
\Gam^{(r)}(\widehat{\omega})=\left\{r^{j+k}\widehat{\omega}(j+k)\right\}_{j,k\geq 0}.
\end{gather}
\noindent Note that for $r=e^{-1/N}$, the above reduces to the truncation considered in \eqref{eqn:poistrun}. For $0<r<1$, the map $H(\omega)\mapsto H(\omega_{r})$ has the following properties 
\begin{enumerate}[label=(\roman*)]
\item for any bounded Hankel operator $H(\omega)$, $H(\omega_{r})\in \frakS_{0}$. Furthermore, \eqref{eqn:split} and H\"older's inequality for Schatten classes (see \cite[Thm. 2, Ch. 11.4]{BirSol1}) give  for $1\leq p\leq \infty$ $$\|H(\omega_{r})\|_{\Sp}\leq \frac{1}{(1-r^{2p})^{1/p}}\|H(\omega)\|;$$
\item if $H(\omega) \in \Sp$ for some $1\leq p\leq \infty$, then \eqref{eqn:split} implies $\|H(\omega_{r})\|_{\Sp}\leq \|H(\omega)\|_{\Sp}.$
\end{enumerate}
\end{subsection}
%%%%%%%%%%%%%%%%%%%%%%%%%%%%%%%%%%%%%%%%%%%%%%%%%%%%%%%%%%%%%%%%%
\begin{subsection}{Almost Orthogonal and Almost Symmetric Hankel operators}
Recall that for a function  $\eta:\bbT\to \bbC$,  its \textit{singular support}, denoted $\singsupp{\eta}$, is defined as the smallest closed subset, $\M$, of $\bbT$ such that $\eta \in C^{\infty}(\bbT\backslash \M)$. 
\begin{lemma}\label{lemma:lemma4.3}
The following statements hold
\begin{enumerate}[label=(\roman*)]
\item \label{item:item4.3.1} Let $\omega_{1}, \omega_{2}\, \in L^{\infty}(\bbT)$ have disjoint singular supports. Set $(\omega_{i})_{r}=P_{r}\ast \omega_{i}, i=1,2$. Then
\begin{align*}
&\sup_{r<1}\|H((\omega_{1})_{r})^{\ast}H((\omega_{2})_{r})\|_{\frakS_{1}}<\infty,\quad \sup_{r<1}\|H((\omega_{1})_{r})H((\omega_{2})_{r})^{\ast}\|_{\frakS_{1}}<\infty;
\end{align*}
\item \label{item:item4.3.2} Suppose $\omega \in L^{\infty}(\bbT)$ be such that $\pm 1 \notin \singsupp(\omega).$ Let $\fraks(v)=\sign(\Im(v)),\, v\in \bbT$. Then $$\sup_{r<1}\|\fraks H(\omega_r)+H(\omega_r)\fraks\|_{\frakS_{1}}<\infty.$$ 
\end{enumerate}
\end{lemma}

\begin{remark}
Similar results are already known in the literature, but only from a qualitative standpoint. In fact, under the same assumptions of \ref{item:item4.3.1}, it is known that both $H(\omega_{1})^{\ast}H(\omega_{2})$ and $H(\omega_{2})^{\ast}H(\omega_{1}) \in \frakS_{0}$. Similarly for \ref{item:item4.3.2}, it is also known that $\fraks H(\omega)+H(\omega)\fraks \in \frakS_{0}$. For a proof of both facts see \cite[Lemma 2.5]{P-Y-Localization} and \cite[Lemma 4.2]{P-Y-Asymptotics} respectively, even though both facts are already mentioned in \cite{Pow-Hankel}.
\end{remark}
\noindent To prove the statements in Lemma \ref{lemma:lemma4.3}, we use the following:
\begin{lemma}\label{lemma:lemma4.1}\ \begin{enumerate}[label=(\roman*)]
\item\label{item:item4.1.1} if $K$ is an operator on $\Ltwo$ with integral kernel $k \in C^{\infty}(\bbT^{2})$, then $K \in \frakS_{1}$;
\item\label{item:item4.1.2} $H(\omega) \in \frakS_{1}$ if $\omega \in C^{2}(\bbT)$ and furthermore there exists $C>0$ such that:$$\|H(\omega)\|_{\frakS_{1}}\leq \|\omega\|_{\infty}+C\|\omega''\|_{2}.$$
\item\label{item:item4.1.3} if $\omega \in C^{2}(\bbT)$, the commutator $[P_{+}, \omega]$ is trace-class.
\end{enumerate}
\end{lemma}
\begin{proof}[Proof of Lemma \ref{lemma:lemma4.1}]
\ref{item:item4.1.1} is folklore. It can be proved by approximating the kernel $k$ by trigonometric polynomials.

\noindent Let us prove \ref{item:item4.1.2}. First, recall two facts:
\begin{enumerate}[label=(\alph*)]
\item any $\omega \in C^{2}(\bbT)$ is the uniform limit of the sequence $$\omega_{N}(v)=\sum_{\abs{j}\leq N-1}{\widehat{\omega}(j)v^{j}},\ \ \ v \in \bbT.$$ Thus for any $N$ and $v \in \bbT$, the Cauchy-Schwarz inequality together with Plancherel Identity give:
\begin{align*}
\left|\omega(v)-\omega_{N}(v)\right|&\leq\|\omega''\|_{2}\left(2\sum_{j\geq N}j^{-4}\right)^{1/2}\leq C  \|\omega''\|_{2} N^{-3/2}.
\end{align*}
\item for $A \in \Sinfty$ one has $s_{N}(A)=\inf\{\|A-B\|\  |\ rank(B)\leq N-1\},\ N\geq 1,$ and, in particular, $s_{1}(A)=\|A\|.$
\end{enumerate}
Putting these two facts together and noting that $rank\left(H(\omega_{N})\right)\leq N$, we have that for $N\geq 2$:
\begin{align*}
s_{N}(H(\omega))&=\inf\{\|H(\omega)-B\|\ |\ rank(B)\leq N-1\}\\
				&\leq \|H(\omega)-H(\omega_{N-1})\|\\
				&\leq \|\omega-\omega_{N-1}\|_{\infty}\\
				&\leq C\frac{\|\omega''\|_{2}}{(N-1)^{3/2}}.
\end{align*}
Thus we see that $H(\omega) \in \frakS_{1}$ and, furthermore, 
\begin{gather*}
\|H(\omega)\|_{\frakS_{1}}=s_{1}(H(\omega))+\sum_{n\geq 2}{s_{n}(H(\omega))}\leq \|\omega\|_{\infty}+C \|\omega''\|_{2}.
\end{gather*}
\ref{item:item4.1.3}. Write $P_{-}=I-P_{+}$, where $I$ is the identity operator. Since $P_{+}$ is a projection and $P_{+}P_{-}=P_{-}P_{+}=0$, one has $$[P_{+}, \omega]=[P_{+},(P_{+}+P_{-}) \omega](P_{+}+P_{-})=P_{+}\omega P_{-}-P_{-}\omega P_{+}.$$ Using the identity $P_{-}=JP_{+}J-P_{+}JP_{+}$, it follows that $$[P_{+}, \omega]=H(\omega)J-JH(\overline{\omega})^{\ast}-P_{+}\omega P_{+}JP_{+}+P_{+}JP_{+}\omega P_{+}.$$
Since $P_{+}JP_{+}$ is a rank-one operator (projection onto constants), $[P_{+}, \omega]$ is trace-class if and only if $H(\omega)J-JH(\overline{\omega})^{\ast}$ is, which follows immediately from \ref{item:item4.1.2}.
\end{proof}
With these facts at hand, we are now ready to prove Lemma\ref{lemma:lemma4.3}.
\begin{proof}[Proof of Lemma \ref{lemma:lemma4.3}]
\ref{item:item4.3.1}: we will only show the first inequality, as the second can be proved in the same way.  From the assumptions on $\omega_{1}, \omega_{2}$, we can find $\zeta_{1}, \zeta_{2}\ \in\ C^{\infty}(\bbT)$ such that $\supp\zeta_{1}\cap \supp\zeta_{2}=\varnothing$ and such that $(1-\zeta_{i})\omega_{i}$ vanishes identically in a neighbourhood of $\singsupp\omega_{i}$. We will repeatedly use the following two facts:
\begin{enumerate}[label=(\alph*)]
\item \label{item:facta} for any $\phi \in L^\infty(\bbT)$, Young's inequality holds, i.e one has the estimate:
\begin{gather}\label{eqn:youngs}
\|P_{r}\ast\phi\|_{\infty}\leq \|\phi\|_{\infty};
\end{gather}
\item \label{item:factb} one has that $P_{r}\ast \omega \in C^{\infty}(\bbT)$ and furthermore $(P_{r}\ast \omega) \to \omega$ as $r\to 1-$ locally uniformly on $\bbT \setminus \singsupp{\omega}$. The same is true for its derivatives $(P_{r}\ast \omega)^{(n)}$. 
\end{enumerate}

\noindent We set $\widetilde{\zeta}_{i}=1-\zeta_{i},\, i=1,2$ and use the triangle inequality to obtain
\begin{align*}
\|H((\omega_{1})_{r})^{\ast}H((\omega_{2})_{r})\|_{\frakS_{1}}&\leq\|H(\widetilde{\zeta}_{1}(\omega_{1})_{r})^{\ast}H(\widetilde{\zeta}_{2}(\omega_{2})_{r})\|_{\frakS_{1}}
														 +\|H(\widetilde{\zeta}_{1}(\omega_{1})_{r})^{\ast}H(\zeta_{2}(\omega_{2})_{r})\|_{\frakS_{1}}\\
														 &+\|H(\zeta_{1}(\omega_{1})_{r})^{\ast}H(\widetilde{\zeta}_{2}(\omega_{2})_{r})\|_{\frakS_{1}}
														 +\|H(\zeta_{1}(\omega_{1})_{r})^{\ast}H(\zeta_{2}(\omega_{2})_{r})\|_{\frakS_{1}},
\end{align*}
from which we see that it is sufficient to find uniform bounds for each summand above.

\noindent Recall that $H(\omega_{i})=P_{+}\omega_{i}JP_{+}$ and $P_{+}$ is a projection, thus: $$H(\zeta_{1}(\omega_{1})_{r})^{\ast}H(\zeta_{2}(\omega_{2})_{r})=P_{+}J\overline{\zeta_{1}(\omega_{1})_{r}}P_{+}\zeta_{2}(\omega_{2})_{r}JP_{+}.$$ Since $\zeta_{1}$ and $\zeta_{2}$ have disjoint supports, the operator $\overline{\zeta}_{1}P_{+}\zeta_{2}$ has a $C^{\infty}(\bbT^{2})$ integral kernel given by $$\frac{\overline{\zeta}_{1}(z)\zeta_{2}(v)}{v-z}v,\quad v,\, z \in \bbT .$$ Lemma \ref{lemma:lemma4.1}-\ref{item:item4.1.1} shows that $\overline{\zeta}_{1}P_{+}\zeta_{2} \in \frakS_{1}$. Furthermore, using H\"older inequality for the Schatten classes and \eqref{eqn:youngs}, we deduce that
\begin{align*}
\sup_{r<1}\|H(\zeta_{1}(\omega_{1})_{r})^{\ast}H(\zeta_{2}(\omega_{2})_{r})\|_{\frakS_{1}}\leq\|\omega_{1}\|_{\infty}\|\omega_{2}\|_{\infty}\|\overline{\zeta}_{1}P_{+}\zeta_{2}\|_{\frakS_{1}}<\infty.
\end{align*} 

\noindent By Lemma \ref{lemma:lemma4.1}-\ref{item:item4.1.2}, we also have that $H(\widetilde{\zeta}_{1}(\omega_{1})_{r}) \in \frakS_{1}$ and furthermore: 
\begin{align}\notag
\|H(\widetilde{\zeta}_{1}(\omega_{1})_{r})^{\ast}H(\zeta_{2}(\omega_{2})_{r})\|_{\frakS_{1}}&\leq\|H(\widetilde{\zeta}_{1}(\omega_{1})_{r})\|_{\frakS_{1}}\|H(\zeta_{2}(\omega_{2})_{r})\| \\ \label{eqn:estimate}
&\leq  \|\zeta_{2}\|_{\infty}\|\omega_{2}\|_{\infty}(C\|(\widetilde{\zeta}_{1}(\omega_{1})_{r})''\|_{2}+\|\omega_{1}\|_{\infty}),
\end{align}
for some $C>0$ independent of $r$. In \eqref{eqn:estimate} we used once more the H\"older inequality for Schatten classes together with the estimates \eqref{eqn:normestimate} and \eqref{eqn:youngs}. 

From \ref{item:factb} and the fact that $\widetilde{\zeta_{i}}\omega_{i}$ vanishes identically on $\singsupp{\zeta_{i}}$, we conclude that $(\widetilde{\zeta}_{i}(\omega_{i})_{r})''\to (\widetilde{\zeta}_{i}\omega_{i})''$ uniformly on the whole of $\bbT$, and so
\begin{align}\label{eqn:uniformderivative}
\sup_{r<1}\|(\widetilde{\zeta}_{i}(\omega_{i})_{r})''\|_{2}<\infty.
\end{align}
Using \eqref{eqn:uniformderivative} in \eqref{eqn:estimate} finally gives $$\sup_{r<1}{\|H(\widetilde{\zeta}_{1}(\omega_{1})_{r})^{\ast}H(\zeta_{2}(\omega_{2})_{r})\|_{\frakS_{1}}}<\infty.$$
Similarly one can show that
\begin{align*}
\sup_{r<1}{\|H(\widetilde{\zeta_{1}}(\omega_{1})_{r})^{\ast}H(\widetilde{\zeta}_{2}(\omega_{2})_{r})\|_{\frakS_{1}}}<\infty,\quad \sup_{r<1}{\|H(\zeta_{1}(\omega_{1})_{r})^{\ast}H(\widetilde{\zeta}_{2}(\omega_{2})_{r})\|_{\frakS_{1}}}<\infty.
\end{align*}
\ref{item:item4.3.2} Since $\pm 1 \notin \singsupp{\omega}$, then we can write $\omega=\phi+\eta$ for some $\eta \in C^{\infty}(\bbT)$ and some $\phi$ vanishing identically in a neighbourhood $U$ of $\pm 1$. With this decomposition of $\omega$, we can see that $$H(\omega_{r})=H(\phi_{r})+H(\eta_{r}).$$ Since $\eta$ is smooth, then $H(\eta) \in \frakS_{1}$ and so the triangle inequality and H\"older inequality for Schatten classes imply that $$\sup_{r<1}\|\fraks H(\omega_{ r})+H(\omega_{r})\fraks\|_{\frakS_{1}}\leq 2 \|H(\eta)\|_{\frakS_{1}}+\sup_{r<1}\|\fraks H(\phi_{ r})+H(\phi_{r})\fraks\|_{\frakS_{1}}.$$ So it is sufficient to consider those symbols $\omega$ vanishing on a neighbourhood, $U$, of $\pm 1$.

\noindent Fix a smooth function $\zeta$ such that $0\leq \zeta\leq 1$,  it vanishes identically on some open $V\subset U$ so that $\pm 1 \in V$, $\zeta \equiv 1$ on $\bbT\backslash U$ and such that $\zeta(v)=\zeta(\overline{v}),\, v\in \bbT$. We can write:
\begin{align}\notag
\fraks H(\omega_r)+H(\omega_r)\fraks&=\fraks H((1-\zeta)\omega_r)+H((1-\zeta)\omega_r)\fraks\\ \label{eqn:decomposition}
									&+\fraks H(\zeta\omega_r)+H(\zeta\omega_r)\fraks
\end{align}
Let us study these operators  more closely.
Using the triangle inequality, we obtain that 
\begin{equation}\label{eqn:Ar1}
\sup_{r<1}\|\fraks H((1-\zeta)\omega_r)+H((1-\zeta)\omega_r)\fraks\|_{\frakS_{1}}\leq 2\sup_{r<1}\|H((1-\zeta)\omega_r)\|_{\frakS_{1}}.
\end{equation}
Using \ref{item:factb} and the fact that $(1-\zeta)\omega \equiv 0$ on $\bbT$, we conclude that $((1-\zeta)\omega_{r})''\to 0$ on $\bbT$ and so Lemma \ref{lemma:lemma4.1}-\ref{item:item4.1.2} gives
\begin{equation}\label{eqn:Ar2}
\sup_{r<1}\|H((1-\zeta)\omega_{r})\|_{\frakS_{1}}\leq \sup_{r<1}(\|\omega\|_{\infty}+C\|((1-\zeta)\omega_{r})''\|_{2})<\infty. 
\end{equation}
For the operators appearing in the second line of \eqref{eqn:decomposition}, write $$\fraks H(\zeta\omega_r)+H(\zeta\omega_r)\fraks=\left(\left[\fraks, P_{+}\right]\zeta\right)\omega_{r}JP_{+}+P_{+}\omega_{r}J\left(\zeta\left[P_{+}, \fraks\right]\right).$$ Let us now prove that the commutators $\left[\fraks, P_{+}\right]\zeta,\, \zeta\left[\fraks, P_{+}\right] \in \frakS_{1}$. By our choice of $\fraks$ and $\zeta$, we have $J\fraks=-\fraks J$ and $J\zeta=\zeta J$, whence it follows that
\begin{align*}
\left[\fraks, P_{+}\right]\zeta&=\fraks P_{+}\zeta-\fraks \zeta P_{+}+\fraks \zeta P_{+}-P_{+}\fraks\zeta=\fraks\left[P_{+}, \zeta\right]+\left[\fraks \zeta, P_{+}\right],\\
\zeta\left[\fraks, P_{+}\right]&=\zeta\fraks P_{+}-P_{+}\fraks \zeta+P_{+}\fraks \zeta -\zeta P_{+}\fraks=\left[\fraks \zeta, P_{+}\right]+\left[P_{+}, \zeta\right]\fraks.
\end{align*}
Furthermore, our choice of $\zeta$ gives that the product $\fraks\zeta \in C^{\infty}(\bbT)$, and Lemma \ref{lemma:lemma4.1}-\ref{item:item4.1.3} together with \eqref{eqn:youngs} implies that
\begin{equation}\label{eqn:Br}
\sup_{r<1}\|\fraks H(\zeta\omega_r)+H(\zeta\omega_r)\fraks\|_{\frakS_{1}}\leq \|\omega\|_{\infty}(\|\left[\fraks, P_{+}\right]\zeta\|_{\frakS_{1}}+\|\zeta\left[P_{+}, \fraks\right]\|_{\frakS_{1}})<\infty.
\end{equation}
Putting together \eqref{eqn:Ar1}, \eqref{eqn:Ar2} and \eqref{eqn:Br} and using the triangle inequality on \eqref{eqn:decomposition} gives the assertion.
\end{proof}

\end{subsection}
%%%%%%%%%%%%%%%%%%%%%%%%%%%%%%%%%%%%%%%%%%%%%%%%%%%%%%%%%%%%%%%%%
\begin{subsection}{Spectral density of our model operator: the Hilbert matrix}
An important ingredient to the proof of all our results is the model operator for which it is possible to explicitly compute the spectral density. Following the ideas of previous works, \cite{Pow-Disc, P-Y-Spectral}, a natural candidate is the Hilbert matrix, given by the symbol $\gamma$ defined in \eqref{eqn:hilbsymb}. Putting together the result of Widom, see \cite[Theorem 5.1]{Widom} and the Invariance Principle \ref{thm:invprinc}, we obtain
\begin{proposition}\label{prop:prop3.1}
 Let $\tau$ satisfy assumptions \ref{(B)} and \ref{(C)}, then one has that $$\uld_{\tau}(t;\Gam(\widehat{\gamma}))=\lld_{\tau}(t;\Gam(\widehat{\gamma}))=\mathsf{c}(t),\ \ \ t>0,$$ where $\mathsf{c}$ has been defined in \eqref{eqn:hilbdensity}. If $\tau(x,y)=\overline{\tau(y,x)}$, then we also have \begin{gather*}
\uld^{+}_{\tau}(t;\Gam(\wh{\gamma}))=\lld^{+}_{\tau}(t;\Gam(\wh{\gamma}))=\mathsf{c}(t),\\ \uld^{-}_{\tau}(t;\Gam(\wh{\gamma}))=\lld^{-}_{\tau}(t;\Gam(\wh{\gamma}))=0.
\end{gather*}
\end{proposition}
\noindent As an immediate consequence of the above, we obtain
\begin{corollary}\label{cor:cor3.3}
Let $z \in \bbT$ be fixed and let $\gamma_{z}(v)=-i\gamma(\overline{z}v)$. Then the same result of Proposition \ref{prop:prop3.1} holds for the operator $\Gam(\widehat{\gamma}_{z})$.
\end{corollary}
\begin{proof}[Proof of Proposition] The Invariance Principle, Theorem \ref{thm:invprinc}, shows that it is sufficient for the statement to hold for $\tau_{\square}(x, y)=\1_{\square}(x,y)$ defined in \eqref{eqn:square}. This has already been done in \cite[Theorem 5.1]{Widom} and it has already been discussed in the Introduction in \eqref{eqn:hilbdensity}. Since the Hilbert matrix is a positive-definite operator, it is easy to see that $\tau_{N}\star \Gam(\wh{\gamma})$ is positive-definite and so 
\begin{align*}
\uld_{\tau}(t;\Gam(\wh{\gamma}))=\uld^{+}_{\tau}(t;\Gam(\wh{\gamma})), \quad \uld^{-}_{\tau}(t;\Gam(\wh{\gamma}))=0.
\end{align*}
The statement can be independently proved using the function $\tau_{1}(x,y)=e^{-(x+y)}$ discussed in the Introduction, however we postpone this to the Appendix.
\end{proof}
\begin{proof}[Proof of Corollary \ref{cor:cor3.3}]
Indeed, note that $\widehat{\gamma}_{z}(j)=-i\,\overline{z}^{j}\,\widehat{\gamma}(j),\, j\geq 0$. Hence, for any function $\tau$ one has: $$\tau_{N}\star\Gam(\widehat{\gamma}_{z})=-iU_{\overline{z}}(\tau_{N}\star\Gam(\widehat{\gamma}))U_{\overline{z}}$$ where $U_{\overline{z}}$ is the unitary operator of multiplication by  $\overline{z}^{j},\, j\geq 0.$ From this, we immediately see that $$s_{n}(\tau_{N}\star\Gam(\widehat{\psi}_{z}))=s_{n}(\tau_{N}\star\Gam(\widehat{\gamma})),\quad \forall\,n\geq 1$$ and so the statement follows immediately from Proposition \ref{prop:prop3.1}.
\end{proof}
\end{subsection}
\end{section}
%%%%%%%%%%%%%%%%%%%%%%%%%%%%%%%%%%%%%%%%%%%%%%%%%%%%%%%%%%%%%%%%%
\begin{section}{Proof of Theorem \ref{thm:sd}}
The proof of the result will be broken down in two Steps. For brevity, we denote by $\Gam^{(N)}(\wh{\omega})$ the operator $\tau_{N}\star\Gam(\wh{\omega})$. We also recall that $\Omega$ is the set of jump-discontinuities of the symbol $\omega$ and $\sfc$ is the function in \eqref{eqn:hilbdensity}.

\noindent \textit{Step 1. Finitely many jumps.} Suppose that $\Omega$ is finite. Setting $\gamma_{z}(v)=-i\gamma(\overline{z}v)$, with $\gamma$ being the symbol defined in \eqref{eqn:hilbsymb}, write
\begin{equation}\label{eqn:1}
\omega(v)=\sum_{z \in \Omega}{\kappa_{z}(\omega)\gamma_{z}(v)}+\eta(v)
\end{equation}
 where $\eta$ is continuous on $\bbT$ and let $\Phi$ denote the symbol $$\Phi(v)=\sum_{z \in \Omega}{\kappa_{z}(\omega)\gamma_{z}(v)}.$$ Weyl's inequality \eqref{eqn:weyl} shows that for $0<s<t$ one has
\begin{gather*}
\sfn(t+s;\Gam^{(N)}(\wh{\Phi}))-\sfn(s;\Gam^{(N)}(\wh{\eta}))\leq \sfn(t;\Gam^{(N)}(\wh{\omega}))\leq \sfn(t-s;\Gam^{(N)}(\wh{\Phi}))+\sfn(s;\Gam^{(N)}(\wh{\eta})).
\end{gather*}
Since $\Gam(\wh{\eta})$ is compact, Lemma \ref{lemma:compactperturbations} shows that $\sfn(s;\Gam^{(N)}(\wh{\eta}))=O_{s}(1)$ as $N\to \infty$ and so, using the definition of the functionals $\lld_{\tau}, \uld_{\tau}$ we deduce that for any $t>0$
\begin{gather}\label{eqn2}
\uld_{\tau}(t; \Gam(\wh{\omega}))\leq\uld_{\tau}(t-0; \Gam(\wh{\Phi})), \\\label{eqn3}
\lld_{\tau}(t; \Gam(\wh{\omega}))\geq\lld_{\tau}(t+0; \Gam(\wh{\Phi})). 
\end{gather} 
Integration by parts shows that 
\begin{align}\notag
\wh{\Phi}(j)&=\sum_{z \in \Omega}{\kappa_{z}(\omega)\wh{\gamma}_{z}(j)}\\ \label{eqn:fouriercoeff}
&=\frac{-i}{\pi (j+1)}\sum_{z \in \Omega}{\kappa_{z}(\omega)z^{j}}=O\left(\frac{1}{j+1}\right),\ \ \ j\to \infty
\end{align}
and so by the Invariance Principle \ref{thm:invprinc} applied to the operator $\Gam(\wh{\Phi})$, we obtain
\begin{gather}\label{eqn4}
\uld_{\tau}(t; \Gam(\wh{\Phi}))\leq\uld_{\tau_{1}}(t-0; \Gam(\wh{\Phi})), \\\label{eqn5}
\lld_{\tau}(t; \Gam(\wh{\Phi}))\geq\lld_{\tau_{1}}(t+0; \Gam(\wh{\Phi})). 
\end{gather} 
where $\tau_{1}(x,y)=e^{-(x+y)}$ induces the regularisation in \eqref{eqn:poistrun}. The Fourier transform $\calF$ on $\Ltwo$, defined as $$(\calF f)(j)=\int_{\bbT}f(z)\overline{z}^{j}d\m(z),\quad f \in \Ltwo,\, j\geq 0,$$ implies that modulo kernels, see \eqref{eqn:poissontruncation}, we have
\begin{gather*}
\Gam^{(N)}(\wh{\Phi})=\sum_{z\in \Omega}\kappa_{z}(\omega)\Gam^{(N)}(\wh{\gamma}_{z})=\sum_{z \in \Omega}{\kappa_{z}(\omega)\calF H((\gamma_{z})_{N})\calF^{\ast}},
\end{gather*}
where $(\gamma_{z})_{N}=P_{r}\ast \gamma_{z}$, with $P_{r}$ being the Poisson kernel with $r=e^{-1/N}$. By Lemma \ref{lemma:lemma4.3}-\ref{item:item4.3.1} and unitary equivalence, we have that whenever $z\neq w$
\begin{gather*}
\sup_{N\geq 1}\|\Gam^{(N)}(\wh{\gamma}_{z})^{\ast}\Gam^{(N)}(\wh{\gamma}_{w})\|_{\frakS_{1}}=\sup_{N\geq 1}\|H((\gamma_{z})_{N})^{\ast}H(\gamma_{w})_{N}\|_{\frakS_{1}}<\infty.
\end{gather*}
Using Theorem \ref{thm:thm2.2}, it then follows that for $t>0$
\begin{gather}\label{eqn6}
\uld_{\tau_{1}}(t; \Gam(\wh{\Phi}))\leq\sum_{z \in \Omega}\uld_{\tau_{1}}(t-0; \kappa_{z}(\omega)\Gam(\wh{\gamma}_{z})), \\\label{eqn7}
\lld_{\tau_{1}}(t; \Gam(\wh{\Phi}))\geq\sum_{z\in \Omega}\lld_{\tau_{1}}(t+0; \kappa_{z}(\omega)\Gam(\wh{\gamma}_{z})). 
\end{gather} 
Finally, Corollary \ref{cor:cor3.3} together with \eqref{eqn2}, \eqref{eqn3}, \eqref{eqn4}, \eqref{eqn5}, \eqref{eqn6} and \eqref{eqn7} and the continuity of $\sfc$ at $t\neq 0$ gives that
\begin{gather*}
\uld_{\tau}(t;\Gam(\wh{\omega}))\leq \sum_{z \in \Omega}\mathsf{c}\left(\frac{t}{\abs{\kappa_{z}(\omega)}}\right),\cr
\uld_{\tau}(t;\Gam(\wh{\omega}))\geq \sum_{z\in \Omega}\mathsf{c}\left(\frac{t}{\abs{\kappa_{z}(\omega)}}\right).
\end{gather*}
The obvious inequality $\lld_{\tau}(t;H(\omega))\leq \uld_{\tau}(t;H(\omega))$ proves the assertion.
\begin{remark}
We note that \eqref{eqn6} and \eqref{eqn7} hold if we consider any symbol $\omega$ which is smooth except for a finite set of jumps discontinuities. These two together are yet another instance of the Localisation principle we referred to in the Introduction.
\end{remark}
\noindent\textit{Step 2. From finitely many to infinitely jumps.} Suppose now that $\Omega$ is infinite. Define the sets:
\begin{gather*}
\Omega_{0}=\{z \in \bbT\ |\ \abs{\kappa_{z}(\omega)}\geq 2^{-1} \},\\
\Omega_{n}=\{z \in \bbT\ |\ 2^{-n-1}\leq\abs{\kappa_{z}(\omega)}< 2^{-n} \},\ \ \ n\geq 1.
\end{gather*}
As we mentioned earlier, these are finite. Let $\phi_{n}$ be functions such that $\singsupp\phi_{n}=\Omega_{n}$, $\kappa_{z}(\phi_{n})=\kappa_{z}(\omega)$ for any $z \in \Omega_{n}$ and such that $$\|\phi_{n}\|_{\infty}=\max_{z \in \Omega_{n}}\abs{\kappa_{z}(\omega)}.$$
Let $\Phi=\sum_{n\geq 0}{\phi_{n}} \in L^{\infty}(\bbT)$. Since $\omega - \Phi\in C(\bbT)$, the operator $\Gam(\wh{\omega})-\Gam(\wh{\Phi}) \in \Sinfty$ and so, by Lemma \ref{lemma:compactperturbations} once again we obtain
\begin{gather*}
\uld_{\tau}(t; \Gam(\wh{\omega}))\leq \uld_{\tau}(t-0; \Gam(\wh{\Phi})),\\
\uld_{\tau}(t; \Gam(\wh{\omega}))\geq \uld_{\tau}(t+0; \Gam(\wh{\Phi})).
\end{gather*}
For a fixed $s>0$, let $M$ be so that $\|\Phi-\Phi_{M}\|_{\infty}< s$, where $\Phi_{M}=\sum_{n=0}^{M}{\phi_{n}}$. The uniform boundedness of $\tau$ then gives $$\|\tau_{N}\star(\Gam(\wh{\Phi})-\Gam(\wh{\Phi}_{M}))\|\leq \left(\sup_{N\geq 1}\|\tau_{N}\|_{\frakM}\right)\|\Phi-\Phi_{M}\|_{\infty}< \left(\sup_{N\geq 1}\|\tau_{N}\|_{\frakM}\right) s:=s'.$$ Letting $\widetilde{\Omega}_{M}=\cup_{n=0}^{M}\Omega_{n}$, we then obtain that:
\begin{gather*}
\uld_{\tau}(t; H(\omega))\leq \uld_{\tau}(t-s'; H(\Phi_{M}))=\sum_{z\in \widetilde{\Omega}_{M}}\sfc\left(\frac{t-s'}{\abs{\kappa_{z}(\omega)}}\right),\\
\uld_{\tau}(t; H(\omega))\geq \uld_{\tau}(t+s'; H(\Phi_{M}))=\sum_{z\in \widetilde{\Omega}_{M}}\sfc\left(\frac{t+s'}{\abs{\kappa_{z}(\omega)}}\right).
\end{gather*}
The equalities above follow from the \textit{Step 1.}, since $\Phi_{M}$ has finitely many jumps. Finally, sending $s \to 0$ and noting that there are only finitely many $z \in \Omega$ such that $t\leq \abs{\kappa_{z}(\omega)}$, one obtains 
\begin{equation*}
\lld_{\tau}(t; H(\omega))=\uld_{\tau}(t; H(\omega)).
\end{equation*}
\end{section}
%%%%%%%%%%%%%%%%%%%%%%%%%%%%%%%%%%%%%%%%%%%%%%%%%%%%%%%%%%%%%%%%%
\begin{section}{Proof of Theorem \ref{thm:sdsa}}
Just as in the proof of Theorem \ref{thm:sd}, we break the argument into two steps, and use the same notation as before for the operator $\tau_{N}\star\Gam(\wh{\omega})$ and for the symbols $\gamma_{z}$. We also set $\Omega^{+}=\{z \in \Omega\ \vert\ \im z>0\}$.

\noindent\textit{Step 1. Finitely many jumps.} Just as before, suppose that the symbol $\omega$ has finitely-many jump-discontinuities. Write
\begin{equation}\label{eqn5.1}
\omega(v)=\left(\kappa_{1}(\omega)\gamma_{1}(v)+\kappa_{-1}(\omega)\gamma_{-1}(v)+\sum_{z \in \Omega^{+}}{\kappa_{z}(\omega)\gamma_{z}(v)+\overline{\kappa_{z}(\omega)}\gamma_{\overline{z}}}(v)\right)+\eta(v),
\end{equation}
where $\eta$ is continuous on $\bbT$. If $\omega$ has no jump at $\pm 1$, the corresponding quantities do not appear in the above. Denoting by $\Phi$ the sum in the brackets, Weyl inequality \eqref{eqn:weylsa} gives for $0<s<t$
\begin{gather*}
\sfn_{\pm}(t+s;\Gam^{(N)}(\wh{\Phi}))-\sfn(s;\Gam^{(N)}(\wh{\eta}))\leq \sfn_{\pm}(t;\Gam^{(N)}(\wh{\omega}))\leq \sfn_{\pm}(t-s;\Gam^{(N)}(\wh{\Phi}))+\sfn_{\pm}(s;\Gam^{(N)}(\wh{\eta})).
\end{gather*}
By Lemma \ref{lemma:compactperturbations}, we obtain that $\sfn_{\pm}(s;\Gam^{(N)}(\wh{\eta}))=O_{s}(1)$, and so, for any $t>0$, it follows that
\begin{gather*}
\uld_{\tau}^{\pm}(t; \Gam(\wh{\omega}))\leq\uld_{\tau}^{\pm}(t-0; \Gam(\wh{\Phi})), \\
\uld_{\tau}^{\pm}(t; \Gam(\wh{\omega}))\geq\uld_{\tau}^{\pm}(t+0; \Gam(\wh{\Phi})). 
\end{gather*} 
Integration by parts once again shows that 
\begin{align*}
\wh{\Phi}(j)=O\left(\frac{1}{j+1}\right),\quad j\to \infty.
\end{align*}
Applying Theorem \ref{thm:invprinc} to $\Gam(\wh{\Phi})$, we see that it is sufficient to prove the result for the multiplier $\tau(x,y)=e^{-(x+y)}.$ Since the symbols $$\kappa_{1}(\omega)\gamma_{1},\ \kappa_{-1}(\omega)\gamma_{-1},\ \kappa_{z}(\omega)\gamma_{z}+\overline{\kappa_{z}(\omega)}\gamma_{\overline{z}}$$ have mutually disjoint singular supports for $z \in \Omega^{+}$, Lemma \ref{lemma:lemma4.3}-\ref{item:item4.3.2} and Theorem \ref{thm:thm2.2} imply that for $t>0$
\begin{align}\notag
\uld^{\pm}_{\tau}(t; \Gam(\wh{\Phi}))&\leq\uld^{\pm}_{\tau}(t-0; \kappa_{1}(\omega)\Gam(\wh{\gamma_{1}}))+
\uld^{\pm}_{\tau}(t-0; \kappa_{1}(\omega)\Gam(\wh{\gamma_{1}}))\\ \label{eqn:upperdensity}
&+\sum_{z \in \Omega^{+}}\uld^{\pm}_{\tau}(t-0; \kappa_{z}(\omega)\Gam(\wh{\gamma}_{z})+\overline{\kappa_{z}(\omega)}\Gam(\wh{\gamma}_{\overline{z}})) \\\notag
\uld^{\pm}_{\tau}(t; \Gam(\wh{\Phi}))&\geq\uld^{\pm}_{\tau}(t+0; \kappa_{1}(\omega)\Gam(\wh{\gamma_{1}}))+
\uld^{\pm}_{\tau}(t+0; \kappa_{1}(\omega)\Gam(\wh{\gamma_{1}}))\\\label{eqn:lowerdensity}
&+\sum_{z \in \Omega^{+}}\uld^{\pm}_{\tau}(t+0; \kappa_{z}(\omega)\Gam(\wh{\gamma}_{z})+\overline{\kappa_{z}(\omega)}\Gam(\wh{\gamma}_{\overline{z}})) .
\end{align}
The operators $\kappa_{\pm 1}(\omega)\Gam(\wh{\gamma}_{\pm 1})$ are sign definite, and furthermore one has that 
\begin{gather*}
\kappa_{\pm 1}(\omega)\Gam(\wh{\gamma}_{\pm 1})\geq 0\ (\text{resp.} \leq 0)\  \text{if} \ -i \kappa_{\pm 1}(\omega)\geq 0\ (\text{resp.} \leq 0).
\end{gather*}
In either case, Proposition \ref{prop:prop3.1} gives that
\begin{align}\notag
\uld^{\pm}_{\tau}(t;\kappa_{\pm 1}(\omega)\Gam(\wh{\gamma}_{\pm 1}))&=\1_{\pm}(-i \kappa_{\pm 1}(\omega))\uld_{\tau}(t;\kappa_{\pm 1}(\omega)\Gam(\wh{\gamma}_{\pm 1}))\\\label{eqn:density1}
&=\1_{\pm}(-i \kappa_{\pm 1}(\omega))\mathsf{c}\left(t\abs{\kappa_{\pm 1}(\omega)}^{-1}\right)
\end{align}
where $\1_{\pm}$ is the indicator function of $\bbR_{\pm}=(0, \pm \infty)$.

\noindent From Lemma \ref{lemma:lemma4.3}-\ref{item:item4.3.2}, Theorem \ref{thm:thm2.3} and Theorem \ref{thm:sd} above, we get that for any $z \in \Omega^{+}$
\begin{align}\notag
\uld^{\pm}_{\tau}\left(t; \kappa_{z}(\omega)\Gam(\wh{\gamma}_{z})+\overline{\kappa_{z}(\omega)}\Gam(\wh{\gamma}_{\overline{z}})\right)&=\frac{1}{2}\uld_{\tau}\left(t; \kappa_{z}(\omega)\Gam(\wh{\gamma}_{z})+\overline{\kappa_{z}(\omega)}\Gam(\wh{\gamma}_{\overline{z}})\right)\\\label{eqn:densityzeta}
&=\mathsf{c}\left(t\abs{\kappa_{z}(\omega)}^{-1}\right).
\end{align}
Using \eqref{eqn:density1} and \eqref{eqn:densityzeta} in \eqref{eqn:upperdensity} and \eqref{eqn:lowerdensity}, the continuity of $\mathsf{c}$ at $t\neq 0$ gives that 
\begin{gather*}
    \lld^{\pm}_{\tau}(t; \Gam(\wh{\omega}))=\uld^{\pm}_{\tau}(t; \Gam(\wh{\Phi}))
\end{gather*}
and so we arrive at \eqref{eqn:lsdsa}.

\begin{remark}
As we wrote earlier in the Introduction, if the symbol has a pair of complex conjugate jumps, then \eqref{eqn:densityzeta} shows that the upper and lower logarithmic spectral density of $\Gam(\wh{\omega})$ contribute equally to the logarithmic spectral density of $\abs{\Gam(\wh{\omega})}$. This is  an effect of the Symmetry Principle we referred to in the Introduction.
\end{remark}

\noindent\textit{Step 2. From finitely many to infinitely many jumps.} For fixed $s > 0$, define the set $$\Omega_{s}^{+}=\{z \in \Omega\ \vert\ \abs{\kappa_{z}(\omega)}>s \text{ and } \Im z> 0\}.$$ Just as in \textit{Step 2.} in the proof of Theorem \ref{thm:sd}, we can find a symbol $\omega_{s} \in PC(\bbT)$ so that $\|\omega-\omega_{s}\|_{\infty}< s$, the set of its discontinuities is precisely $\Omega_{s}^{+}\cup\{\pm 1\}$ and $$\kappa_{z}(\omega)=\kappa_{z}(\omega_{s}),\quad \forall z \in \Omega_{s}^{+}\cup\{\pm 1\}.$$ The set $\Omega_{s}^{+}\cup\{\pm 1\}$ is finite, thus from Weyl inequality \eqref{eqn:weylsa} and \textit{Step 1.} we obtain
\begin{align*}
\uld^{\pm}_{\tau}(t; \Gam(\wh{\omega}))&\leq \uld^{\pm}_{\tau}(t-s'; \Gam(\wh{\omega}_{s}))\\
\uld^{\pm}_{\tau}(t; \Gam(\wh{\omega}))&\geq \uld^{\pm}_{\tau}(t+s'; \Gam(\wh{\omega}_{s})),
\end{align*}
where $s'=\left(\sup_{N\geq 1}\|\tau_{N}\|_\frakM\right) s.$ Finally, sending $s \to 0$ and using the continuity of $\mathsf{c}$ at $t\neq 0$ establishes the result in its generality. \qed

\begin{proof}[Proof of Proposition \ref{prop:schattenld}]
The same reasoning of \textit{Step 1.} in both proofs above applies in this case, with only one minor change. Since we assume that $\tau$ induces a uniformly bounded multiplier on $\Sp,\, p>1$, i.e. that \eqref{eqn:spuniform} holds, in \eqref{eqn:1} and \eqref{eqn5.1} we need to assume that $\eta$ is a symbol so that $\Gam(\wh{\eta}) \in \Sp$. Then Lemma \ref{lemma:compactperturbations} shows that $\sfn(s;\Gam^{(N)}(\wh{\eta}))=O_{s}(1)$ and, in the self-adjoint case $\sfn_{\pm}(s;\Gam^{(N)}(\wh{\eta}))=O_{s}(1)$. The rest follows immediately.
\end{proof}

\begin{proof}[Proof of Proposition \ref{prop:hankelld}]
Exactly the same reasoning of the proofs of Theorems \ref{thm:sd} and \ref{thm:sdsa} above applies in this case, with the only difference being that in this case $\tau$ is no longer inducing a uniformly bounded multiplier on the whole space of bounded operators, just on Hankel matrices. However, all of the terms appearing in the arguments just presented are bounded Hankel operators and so the same arguments apply in this case.
\end{proof}
\end{section}

\begin{appendix}
\section{An independent proof of Proposition \ref{prop:prop3.1}}
By virtue of Theorem \ref{thm:invprinc}, choose the function $\tau_{1}(x,y)=e^{-(x+y)}$, which yields $$((\tau_{1})_{N}\star\Gam(\widehat{\omega}))_{j,k}=e^{-\frac{j+k}{N}}\widehat{\omega}(j+k)=\Gam^{(r)}(\widehat{\omega})_{j,k},\ \ \ r=e^{-1/N},$$ where $\Gam^{(r)}(\wh{\omega})$ is the Poisson truncation in \eqref{eqn:poissontruncation}. We start our proof with the following Lemma, similar to \cite[Lemma 4.1]{F-P-Spectral}:
\begin{lemma}\label{lemma:lemma3.2}
 For any $m \in \bbN$ one has that: $$\Tr \Gam^{(r)}(\widehat{\gamma})^{m}=\frac{\abs{\log(1-r)}}{2\pi}\int_{\bbR}{\left(\frac{1}{r\cosh(\pi\eta)}\right)^{m}d\eta}+o(\abs{\log(1-r)}),\ \ \ r\to 1^{-}.$$
 \end{lemma}
 
 \begin{proof}[Proof of Lemma] 
Let us define the operator $L:\LTwo\to \ltwo$ as follows
\begin{align*}
(L f)(j)&=\frac{1}{\sqrt{\pi}}\int_{0}^{1}{f(s)s^{j}ds},\ \ \ j\geq 0.
\end{align*}
 Its boundedness can be established using the Schur test. A simple calculation yields the identity $\Gam(\widehat{\gamma})=L L^{\ast},$ from which if follows that, with $\Gam^{(r)}(\widehat{\gamma})=\Gam(\widehat{\gamma}_{r})$ 
\begin{align*}
\Gam^{(r)}(\widehat{\gamma})&=\frac{1}{r}L\1_{r}L^{\ast}\\
								&=\frac{1}{r}(L\1_{r})(L\1_{r})^{\ast}
\end{align*}
where $\1_{r}$ is the characteristic function of the interval $(0,r)$ and so one obtains \begin{gather}\label{eqn:A1}
r^{m}\Tr\Gam^{(r)}(\wh{\gamma})^{m}=\Tr\left(\1_{r}L^{\ast}L\1_{r}\right)^{m},
\end{gather}
therefore we only need to compute the latter trace. Recall now that for any bounded operator $X$, there is a unitary equivalence between $XX^{\ast}|_{\ker(XX^{\ast})^{\perp}}$ and $X^{\ast}X|_{\ker(X^{\ast}X)^{\perp}}$. Hence, the trace of $\left(\1_{r}L^{\ast}L\1_{r}\right)^{m}$ and that of $\left(\1_{r}L^{\ast}L\1_{r}\right)^{m}$ coincide. Note however that the operator $L^{\ast}L$ is an operator acting on $\LTwo$ whose integral kernel is: $$k(t,s)=\frac{1}{\pi(1-ts)},\ \ \ t,\,s \in (0,1).$$ Following the procedure described in \cite{Widom}, define the unitary transformation: 
\begin{align*}
U&:\LTwo\to L^{2}(\Rplus)\\
(U f)(x)&=\frac{1}{\cosh(x)}f(\tanh(x)),\ \ \ x>0.
\end{align*}
Then we have $B=U L^{\ast}LU^{\ast}:L^{2}(\Rplus)\to L^{2}(\Rplus)$ is the convolution operator
\begin{align*}
(Bf)(x)&=\int_{\Rplus}{\frac{f(y)}{\pi\cosh(x-y)}dy},\ \ \ x>0.
\end{align*}
In this way, we have reduced our problem to evaluating the trace of the integral operator $(\widetilde{\1}_{r}B\widetilde{\1}_{r})^m$ , where $\widetilde{\1}_{r}$ is the characteristic function of the interval $(0, \arctanh(r))$. By adding $0$ to its spectrum, we also consider $\widetilde{\1}_{r}B\widetilde{\1}_{r}$ as an integral operator acting on $L^{2}(\bbR)$, with integral kernel $$\frac{\widetilde{\1}_{r}(s)\widetilde{\1}_{r}(t)}{\pi\cosh(s-t)},\ \ \ s,t \in \bbR.$$ We now use the following result:
\begin{theorem}[\cite{L-S}]\label{thm:LS}
Let $P$ be an orthogonal projection and $B$ be a bounded operator such that $PB \in \frakS_{2}$. Let $\phi$ be such that $\phi(0)=0$ and $\phi'' \in L^{\infty}(\spec{B})$, then:
\begin{align}\label{eqn:LS}
\abs{\Tr\phi(PBP)-\Tr P\phi(B)P}\leq \|\phi''\|_{L^{\infty}(\spec{B})}\|PB(I-P)\|_{\frakS_{2}}^{2}.
\end{align}
\end{theorem}
Note that $\1_{r}L^{\ast}L \in \frakS_{2}$ for any $r<1$ so by unitary equivalence $\widetilde{\1}_{r}B \in \frakS_{2}$. Furthermore, the operator $B$ is unitarily equivalent, under the Fourier Transform, to the operator of multiplication on $L^{2}(\bbR)$ by the function $$\frac{1}{\cosh(\pi\xi/2)},\ \ \ \xi \in \bbR.$$
Whence we can estimate $\Tr (\widetilde{\1}_{r}B\widetilde{\1}_{r})^m$ by:
\begin{align*}
\Tr\widetilde{\1}_{r}B^m\widetilde{\1}_{r}&=\frac{1}{2\pi}\int_{\bbR}{\widetilde{\1}_{r}(x)dx}\int_{\bbR}{\left(\frac{1}{\cosh(\pi\xi/2)}\right)^{m}d\xi}\\
											  &=\frac{\arctanh(r)}{\pi}\int_{\bbR}{\left(\frac{1}{\cosh(\pi\xi)}\right)^{m} d\xi}\\
											  &=\frac{\abs{\log(1-r)}}{2\pi}\int_{\bbR}{\left(\frac{1}{\cosh(\pi\xi)}\right)^{m}d\xi}+o(\abs{\log(1-r)}),\ \ \ r \to 1^{-}.
\end{align*}
We also have that:
\begin{align*}
\|\widetilde{\1}_{r}B(1-\widetilde{\1}_{r})\|_{\frakS_{2}}^{2}&=\|(\widetilde{\1}_{r}B-B\widetilde{\1}_{r})(1-\widetilde{\1}_{r})\|_{\frakS_{2}}^{2}\leq \|[\widetilde{\1}_{r},B]\|_{\frakS_{2}}^{2}, 
\end{align*}
thus we need to find an estimate for the Hilbert-Schmidt norm of the integral operator $[\widetilde{\1}_{r},B]$, which has integral kernel given by: 
\begin{equation*}
k(t,s)=\frac{\widetilde{\1}_{r}(t)-\widetilde{\1}_{r}(s)}{\pi\cosh(\pi(t-s))},\ \ \ t, s\in \bbR.
\end{equation*}
It follows that
\begin{align*}
\|[\widetilde{\1}_{r},B]\|_{\frakS_{2}}^{2}=\iint_{\bbR^{2}}{k^{2}(t,s)dtds}=\int_{\bbR}{\frac{\phi(z)}{\pi^{2}\cosh^{2}(\pi z)}dz},
\end{align*}
with
\begin{align*}
\phi(z)&=\int_{\bbR}{(\widetilde{\1}_{r}(z+y)-\widetilde{\1}_{r}(y))^{2}dy}\\
	   &= 2 \min\left\{\abs{z}, \arctanh{r}\right\}\leq 2\abs{z}.
\end{align*}
Whereby obtaining that
\begin{align*}
\|[\widetilde{\1}_{r},B]\|_{\frakS_{2}}^{2}\leq C\int_{\bbR}{\frac{\abs{z}}{\cosh^{2}(z)}dz}<\infty.
\end{align*}
Using \eqref{eqn:A1} and \eqref{eqn:LS}, we obtain: $$\Tr (\Gam^{(r)}(\widehat{\gamma}))^{m}=\frac{\abs{\log(1-r)}}{2\pi r^{m}}\int_{\bbR}{\left(\frac{1}{r\cosh(\pi \xi)}\right)^{m}d\xi}+o(\abs{\log(1-r)}),$$ as $r\to 1^{-}$.
\end{proof}
\noindent Proposition \ref{prop:prop3.1} now follows from a two-step approximation argument. In the first stage, using the Weierstarss Approximation theorem and the Lemma \ref{lemma:lemma3.2}, we prove that for any function $\phi \in C^{\infty}_{c}(\Rplus)$ one has that 
\begin{gather}\label{eqn:A3}
\lim_{r\to 1^{-}}\frac{\Tr \phi(\Gam^{(r)}(\widehat{\gamma}))}{\abs{\log(1-r)}}=\frac{1}{2\pi}\int_{\bbR}{\phi\left(\frac{1}{\cosh(\pi \eta)}\right)d\eta}.
\end{gather}
In the second, we set $r=e^{-1/N}$ and we note that we can replace $\abs{\log(1-e^{-1/N})}$ with $\log(N)$ in the limits above and that we can write $$\sfn(t; \Gam^{(N)}(\widehat{\gamma}))=\Tr\1_{(t, 1)}(\Gam^{(N)}(\widehat{\gamma})).$$ Choose sequences $\phi_{n}^{\pm} \in C^{\infty}_{c}(\Rplus)$ for which we have 
\begin{gather*}
0\leq \phi^{-}_{n}(x)\leq \1_{(t, 1)}(x)\leq \phi^{+}_{n}(x)\leq 1,\ \ \ \forall x,
\end{gather*}
and  $\phi_{n}^{\pm}(x)\to \1_{(t,1)}(x)$ pointwise in $x$ as $n \to \infty$. From the properties of $\Tr$ and \eqref{eqn:A3} it follows
 \begin{align*}
 \uld_{\tau}(t; \Gam(\widehat{\gamma}))&\leq \frac{1}{2\pi}\int_{\bbR}{\phi_{n}^{+}\left(\frac{1}{\cosh(\pi \eta)}\right)d\eta},\\
 \uld_{\tau}(t; \Gam(\widehat{\gamma}))&\geq \frac{1}{2\pi}\int_{\bbR}{\phi_{n}^{-}\left(\frac{1}{\cosh(\pi \eta)}\right)d\eta}.
 \end{align*}
 Finally, an application of the Dominated Convergence Theorem gives the result.
\end{appendix}
\section*{Acknowledgement}{The content of this paper forms part of the author's PhD thesis, completed at King's College London under the supervision of Prof. A.B. Pushnitski. The author is also grateful to both Dr. M. Gebert and Dr. N. Miheisi for the numerous helpful discussions and numerous remarks on earlier versions of this paper.}
%%%%%%%%%%%%%%%%%%%%%%%%%%%%%%%%%%%%%%%%%%%%%%%%%%%%%%%%%%%%%%%%%
\bibliographystyle{acm}
\bibliography{biblio}
\end{document}